\documentclass{amsart}
\usepackage{amssymb}
\usepackage{amsmath, amscd}
\usepackage{amsthm}
\usepackage{perpage}
\usepackage[all,cmtip]{xy}
\usepackage{setspace}
\usepackage{amsbsy}
\usepackage[foot]{amsaddr}
\usepackage{url}
\usepackage[T1]{fontenc}
\usepackage{verbatim}
\usepackage{mathrsfs}
\usepackage{ifthen}
\usepackage{blkarray}
\usepackage{multirow}
\usepackage{enumerate}
\usepackage{tikz-cd}
\usepackage{mathtools}
\usepackage[left=2cm,head=13pt,
top=1.5cm,right=2cm, bottom=1.5cm]{geometry}

\AtEndDocument{\bigskip{\footnotesize%
\noindent (Wushi Goldring)
  \textsc{Department of Mathematics, Stockholm University, Stockholm SE-10691, Sweden} \par  
 \noindent \texttt{wushijig@gmail.com} \par
  \addvspace{\medskipamount}
\noindent (Jean-Stefan Koskivirta) \textsc{Department of Mathematics, South Kensington Campus, Imperial College London, London SW7 2AZ, UK} \par 
\noindent \texttt{jeanstefan.koskivirta@gmail.com}
}}

\numberwithin{equation}{subsection}

\newtheorem{theorem}{Theorem}[subsection]
\newtheorem{lemma}[theorem]{Lemma}

\newtheorem{example}[theorem]{Example}
\newtheorem{corollary}[theorem]{Corollary}
\newtheorem{definition}[theorem]{Definition}

\newtheorem{proposition}[theorem]{Proposition}
\newtheorem{condition}[theorem]{Condition}

\newtheorem{thmi}{Theorem}

\newtheorem{cor}[theorem]{Corollary}
\newtheorem*{cor1}{Corollary 1}

\newtheorem*{ques}{Question}

\theoremstyle{remark}

\newtheorem{rmk}[theorem]{Remark}

\title[Zips and flags]{Stratifications of flag spaces and functoriality}

\newcommand{\GZip}{\mathop{\text{$G$-{\tt Zip}}}\nolimits}
\newcommand{\GpZip}{\mathop{\text{$G'$-{\tt Zip}}}\nolimits}

\newcommand{\GoneZip}{\mathop{\text{$G_1$-{\tt Zip}}}\nolimits}
\newcommand{\GtwoZip}{\mathop{\text{$G_2$-{\tt Zip}}}\nolimits}

\newcommand{\GtiloneZip}{\mathop{\text{$\tilde{G}_1$-{\tt Zip}}}\nolimits}
\newcommand{\GLsixZip}{\mathop{\text{$GL_6$-{\tt Zip}}}\nolimits}

\newcommand{\GF}{\mathop{\text{$G$-{\tt ZipFlag}}}\nolimits}
\newcommand{\GpF}{\mathop{\text{$G'$-{\tt ZipFlag}}}\nolimits}

\DeclareMathOperator{\Gal}{Gal}

\newskip\procskipamount
\newskip\interskipamount
\newskip\refskipamount
\newcommand{\procskip}{\vskip\procskipamount}
\newcommand{\interskip}{\vskip\interskipamount}
\newcommand{\refskip}{\vskip\refskipamount}

\newcommand{\procbreak}{\par
   \ifdim\lastskip<\procskipamount\removelastskip
   \penalty-100
   \procskip\fi
   \noindent\ignorespaces}

\newcommand{\titlebreak}{\par%
\ifdim\lastskip<\interskipamount\removelastskip%
\penalty10000%
\interskip\fi%
\noindent}%

\newcommand{\interbreak}{\par%
\ifdim\lastskip<\interskipamount\removelastskip%
\penalty-100%
\interskip\fi%
\noindent\ignorespaces}%

\newcommand{\refbreak}{\par%
\ifdim\lastskip<\refskipamount\removelastskip%
\penalty-100%
\refskip\fi%
\noindent\ignorespaces}%

\newcounter{listcounter}
\newcounter{deflistcounter}
\newcounter{equivcounter}

\newskip{\itemsepamount}
\newskip{\topsepamount}
\newenvironment{assertionlist}{%
  \begin{list}
    {\upshape (\arabic{listcounter})}
    {\setlength{\leftmargin}{18pt}
     \setlength{\rightmargin}{0pt}
     \setlength{\itemindent}{0pt}
     \setlength{\labelsep}{5pt}
     \setlength{\labelwidth}{13pt}
     \setlength{\listparindent}{\parindent}
     \setlength{\parsep}{0pt}
     \setlength{\itemsep}{\itemsepamount}
     \setlength{\topsep}{\topsepamount}
     \usecounter{listcounter}}}
  {\end{list}}

\newenvironment{definitionlist}{%
  \begin{list}
    {\upshape (\alph{deflistcounter})}
    {\setlength{\leftmargin}{18pt}
     \setlength{\rightmargin}{0pt}
     \setlength{\itemindent}{0pt}
     \setlength{\labelsep}{5pt}
     \setlength{\labelwidth}{13pt}
     \setlength{\listparindent}{\parindent}
     \setlength{\parsep}{0pt}
     \setlength{\itemsep}{\itemsepamount}
     \setlength{\topsep}{\topsepamount}
     \usecounter{deflistcounter}}}
  {\end{list}}

\newenvironment{equivlist}{%
  \begin{list}
    {\upshape (\roman{equivcounter})}
    {\setlength{\leftmargin}{18pt}
     \setlength{\rightmargin}{0pt}
     \setlength{\itemindent}{0pt}
     \setlength{\labelsep}{5pt}
     \setlength{\labelwidth}{13pt}
     \setlength{\listparindent}{\parindent}
     \setlength{\parsep}{0pt}
     \setlength{\itemsep}{\itemsepamount}
     \setlength{\topsep}{\topsepamount}
     \usecounter{equivcounter}}}
  {\end{list}}

% ...

\newcommand{\Acal}{{\mathcal A}}

\newcommand{\Ccal}{{\mathcal C}}
\newcommand{\Dcal}{{\mathcal D}}

\newcommand{\Fcal}{{\mathcal F}}
\newcommand{\Gcal}{{\mathcal G}}

\newcommand{\Ical}{{\mathcal I}}
\newcommand{\Jcal}{{\mathcal J}}

\newcommand{\Lcal}{{\mathcal L}}

\newcommand{\Ocal}{{\mathcal O}}
\newcommand{\Pcal}{{\mathcal P}}

\newcommand{\Rcal}{{\mathcal R}}
\newcommand{\Scal}{{\mathcal S}}

\newcommand{\Ucal}{{\mathcal U}}
\newcommand{\Vcal}{{\mathcal V}}
\newcommand{\Wcal}{{\mathcal W}}
\newcommand{\Xcal}{{\mathcal X}}

\newcommand{\Zcal}{{\mathcal Z}}

\newcommand{\gfr}{{\mathfrak g}}

\newcommand{\pfr}{{\mathfrak p}}

\newcommand{\Pfr}{{\mathfrak P}}

\newcommand{\Sfr}{{\mathfrak S}}

\renewcommand{\AA}{\mathbf{A}}

\newcommand{\CC}{\mathbf{C}}

\newcommand{\FF}{\mathbf{F}}
\newcommand{\GG}{\mathbf{G}}
\newcommand{\HH}{\mathbf{H}}

\newcommand{\NN}{\mathbf{N}}

\newcommand{\PP}{\mathbf{P}}
\newcommand{\QQ}{\mathbf{Q}}
\newcommand{\RR}{\mathbf{R}}
\renewcommand{\SS}{\mathbf{S}}

\newcommand{\ZZ}{\mathbf{Z}}

\DeclareMathOperator{\Pic}{Pic}

\newcommand{\Ascr}{{\mathscr A}}

\newcommand{\Dscr}{{\mathscr D}}

\newcommand{\Lscr}{{\mathscr L}}

\newcommand{\Sscr}{{\mathscr S}}

\newcommand{\Vscr}{{\mathscr V}}

\newcommand{\fp}{\mathbf F_p}

\newcommand{\stab}{{\rm Stab}}

\newcommand{\Ad}{\mbox{Ad }}

\newcommand{\Th}{{\rm Th.}}
\newcommand{\Ths}{{\rm Ths.}}
\newcommand{\Rmk}{{\rm Rmk.}}

\newcommand{\Cor}{{\rm Cor.}}

\newcommand{\Lem}{{\rm Lem.}}

\newcommand{\Chap}{{\rm Chap.}}

\newcommand{\Def}{{\rm Def.}}

\newcommand{\Prop}{{\rm Prop.}}

\newcommand{\loccit}{{\em loc.\ cit. }}
\newcommand{\loccitn}{{\em loc.\ cit.}}
\newcommand{\cf}{{\em cf. }}

\newcommand{\diag}{{\rm diag}}

\newcommand{\Shk}{\Sscr_{K}}

\DeclareMathOperator{\Lie}{Lie}

\DeclareMathOperator{\Int}{int}
\DeclareMathOperator{\adj}{ad}

\renewcommand{\div}{{\rm div}}

\DeclareMathOperator{\Ker}{Ker}

\newcommand{\GFZP}{\GF^{(\Zcal, P_0)}}

\author{Wushi Goldring, Jean-Stefan Koskivirta}
\date{September 11, 2017}

\begin{document}

\begin{abstract}
We define stacks of zip flags, which form towers above the stack of $G$-zips of Moonen, Pink, Wedhorn and Ziegler in \cite{Moonen-Wedhorn-Discrete-Invariants}, \cite{Pink-Wedhorn-Ziegler-zip-data} and \cite{PinkWedhornZiegler-F-Zips-additional-structure}. 
A stratification is defined on the stack of zip flags, and principal purity is established under a mild assumption on the underlying prime $p$. We generalize flag spaces of Ekedahl-Van der Geer \cite{EkedahlGeerEO} and relate them to stacks of zip flags. For large $p$, it is shown that strata are affine. We prove that morphisms with central kernel between stacks of G-zips have discrete fibers. This allows us to prove principal purity of the zip stratification for maximal zip data. The latter provides a new proof of the existence of Hasse invariants for Ekedahl-Oort strata of good reduction Shimura varieties of Hodge-type, first proved in \cite{Goldring-Koskivirta-Strata-Hasse}.
\end{abstract}

\maketitle

\section*{Introduction}
This paper is the second in a series on our program, outlined in \cite{Goldring-Koskivirta-Strata-Hasse}, to connect (A) {\em Automorphic Algebraicity}, (B) {\em $G$-Zip Geometricity} and (C) {\em Griffiths-Schmid Algebraicity}. The two main themes motivating our program are those of {\em geometry-by-groups} and {\em characteristic shifting}. For more details on our program and background on (A), (B), (C), see the introduction of \loccit and the references therein.

The specific focus of this paper is to dig deeper into several aspects of (B) which were studied in \loccit More precisely, four interrelated topics are pursued:
\begin{enumerate}
\item The behavior of the zip stratification under functoriality.
\item The construction of stacks $\GFZP$ of zip partial flags, together with a stratification of Ekedahl-Oort type. The stacks $\GFZP$ lie between a stack of $G$-Zips and the stack of zip flags $\GF^{\Zcal}$ defined in \cite{Goldring-Koskivirta-Strata-Hasse} :
$$\GF^{\Zcal} \to \GFZP \to \GZip^\Zcal.$$
\item The construction of {\em group-theoretical Hasse invariants} on the above stacks, under mild bounds on the characteristic. We view the existence of such Hasse invariants in terms of several strengthened notions of {\em purity}.
\item The construction of {\em partial flag spaces}, which lie in-between a Shimura variety of Hodge type and its flag space.
\end{enumerate}

The main results obtained in each of the four areas singled out above are summarized below. One of the technical innovations in this paper over \cite{Goldring-Koskivirta-Strata-Hasse} is the use of zip data of higher exponent $n$ (corresponding to the $n$th power of Frobenius, see \S\ref{catGzips}), while \loccit only considered zip data of exponent $1$. Passing to a higher exponent is a key method in \Th~\ref{th-intro-discrete-fibers} below.

Another improvement is that a criterion is given for the existence of group-theoretical Hasse invariants on a general zip datum, while \loccit only treated those zip data arising from cocharacters. This generalization is necessary to obtain sections on the strata of $\GFZP$.

The above topics have already had numerous applications. Some applications to the coherent cohomology of Shimura varieties, automorphic forms and the Langlands Program were given in \cite{Goldring-Koskivirta-Strata-Hasse}. Applications to the existence of global sections of automorphic vector bundles, in particular proving a conjecture of Diamond on mod $p$ Hilbert modular forms were given in \cite{Goldring-Koskivirta-Diamond-I}. The second author has applied the theory of partial flag spaces developed here to determine the normalization of Ekedahl-Oort strata closures \cite{Koskivirta-Normalization}. For additional applications to ampleness of automorphic bundles and vanishing theorems for the coherent cohomology of Shimura varieties, see forthcoming work with Brunebarbe and Stroh \cite{Brunebarbe-Goldring-Koskivirta-Stroh-ampleness}. We expect the above applications to be merely the tip of the iceberg: many more will follow.

\subsection*{Stacks of zip flags and their stratification}
 Let $G$ be a connected, reductive $\fp$-group; let $\Zcal=(G,P,L,Q,M,\varphi^n)$ be a zip datum  and let $\GZip^\Zcal$ be the associated stack of $G$-Zips (\S\ref{catGzips}). The following summarizes some of the basic constructions in \S\ref{sec stofzf} and \S\ref{subsec def fine strata}.
\begin{thmi}
\label{th-intro-zip-flags}
Let $P_0$ be a sub-parabolic of $P$. 

\begin{enumerate}
\item There exists a smooth stack $\GFZP$ of dimension $\dim(P/P_0)$ of zip flags of type $P_0$. 
\item There is a natural projection $\GFZP \to \GZip^{\Zcal}$, which is smooth with flag variety fibers $P/P_0$.
\item If $B\subset P_0$  is a Borel subgroup, then the projection $\GF^\Zcal \to \GZip^\Zcal$ factors through $\GF^\Zcal \to \GFZP$, with flag variety fibers $P_0/B$.
\item  
\label{item-intro-fine-strata}
There is a second zip datum $\Zcal_0$ associated to $(\Zcal,P_0)$ and a smooth map $\GFZP \to \GZip^{\Zcal_0}$. It induces a stratification of $\GFZP$ parameterized by $^{I_0} W$, where $I_0$ is the type of $P_0$. 
\end{enumerate}
\end{thmi}
The stratification in~\eqref{item-intro-fine-strata} is termed {\em fine}, see \S\ref{subsec def fine strata}. We also define a {\em coarse} stratification of $\GFZP$, see \S\ref{subsec coarse}; it generalizes the Bruhat stratification of $\GZip^{\Zcal}$ studied by Wedhorn \cite{Wedhorn-bruhat}, which in turn
generalizes the a-number stratification studied by Oort and others (\cf introduction of \loccitn).

In general, the zip datum $\Zcal_0$ is not attached to a cocharacter, unless $P_0=P$. There seems to be no known algebraic counterpart of the stack $\GZip^{\Zcal_0}$ in the theory of Shimura varieties. Thus we see that the  group-theoretic approach of $G$-Zips reveals genuinely new structure. Nevertheless, following (i) the analogy between stacks of $G$-Zips and Griffiths-Schmid manifolds suggested in \cite{Goldring-Koskivirta-Strata-Hasse} (see esp. \S I.4) and (ii) the approach to Griffiths-Schmid manifolds in \cite{Green-Griffiths-Kerr-CBMS-Texas}, it appears that there may be a complex-geometric analogue of
$\GZip^{\Zcal_0}$ among the objects studied in \loccitn; see \S\ref{sec-analogue-shimura}.

\Th~\ref{th-intro-zip-flags} generalizes the construction of the stack of zip flags $\GF^{\Zcal}$ in \cite{Goldring-Koskivirta-Strata-Hasse} to an intermediate parabolic $B \subset P_0 \subset P$. In particular, when $P_0=B$ one has $\GFZP=\GF^\Zcal$ and when $P_0=P$ one has $\GFZP=\GZip^\Zcal$. In turn, the case of $\GF^{\Zcal}$ is a group-theoretical generalization of the flag space -- and its stratification -- associated to the moduli space of abelian varieties $\Acal_g \otimes \fp$ by Ekedahl and van der Geer  \cite{EkedahlGeerEO}.

\subsection*{Functoriality of zip stratifications}
The association $\Zcal \to \GZip^\Zcal$ is functorial: A morphism $f:\Zcal_1 \to \Zcal_2$ of zip data is a morphism of the underlying groups $f:G_1 \to G_2$, compatible with the additional structure (see \S\ref{catGzips}). Such a morphims induces one of stacks $\tilde{f}:\GoneZip^{\Zcal_1}\to \GtwoZip^{\Zcal_2}$  (Lemma~\ref{functo}). It is then natural to question how the zip stratification varies under functoriality.

\begin{thmi}[\Th~\ref{discfib}]
\label{th-intro-discrete-fibers}
Let $f:\Zcal_1 \to \Zcal_2$ be a morphism of zip data, such that the underlying morphism of groups $f:G_1 \to G_2$ has central kernel. Then the induced morphism of stacks
\begin{equation*}
\tilde{f}:\GoneZip^{\Zcal_1}\to \GtwoZip^{\Zcal_2}
\end{equation*}
has discrete fibers on the underlying topological spaces.
\end{thmi}
Concretely, the discreteness of fibers means that, if two strata $S$ and $S'$ of $\GoneZip^{\Zcal_1}$ map into the same stratum of  $\GtwoZip^{\Zcal_2}$, then $S$ and $S'$ are incomparable in the zip stratification of $\GoneZip^{\Zcal_1}$ (neither is contained in the closure of the other).  

An immediate application of \Th~\ref{th-intro-discrete-fibers} is that the same type of functoriality is satisfied by the Ekedahl-Oort stratification of Hodge-type Shimura varieties.

\subsection*{(Uniform) principal purity of strata}
Recall that a stratification on a scheme (or stack) $X$ is principally pure if each stratum $S\subset X$ is the non-vanishing locus of a section of a line bundle on the Zariski closure $\overline{S}$. We call such a section a characteristic section for $S$. In case $S$ is a stratum of $\GFZP$, due to the connection with the classical Hasse invariant of an abelian scheme, we also call such a section a group-theoretical Hasse invariant. 

The  following question is central to this work:

\begin{ques}
For which zip data $\Zcal$ is the zip stratification of $\GZip^\Zcal$ principally pure? More generally, for which $(\Zcal,P_0)$ is the fine stratification of $\GF^{(\Zcal, P_0)}$ principally pure?
\end{ques}

An example where the zip stratification is not principally pure is given in \S\ref{counterex}. However, the following  shows that such examples are extremely rare. In fact, the result below shows that a stronger property than principal purity usually holds. We say that a stratification of a scheme (or stack) is {\em uniformly principally pure} if there exists a line bundle that admits characteristic sections for all strata. 

Let $(\Zcal,P_0)$ be a pair as in \Th~\ref{th-intro-zip-flags}, with associated zip datum $\Zcal_0$. For a character $\lambda\in X^*(P_0)$, we define properties "orbitally $q$-close" (\Def~\ref{def-orb-p-close}) and "$\Zcal_0$-ample" (\Def~\ref{defample}). Then one has:

\begin{thmi}[\Th~\ref{thmonw}]
\label{th-intro-unif-purity}
Let $q=p^n$. If there exists a $\Zcal_0$-ample, orbitally $q$-close character of $P_0$, then the fine stratification of $\GF^{(\Zcal,P_0)}$ is uniformly principally pure.
\end{thmi}

\Th~\ref{th-intro-discrete-fibers} and \Th~\ref{th-intro-unif-purity} are closely related. To prove the former, we use the latter as a key ingredient. Conversely, \Th~\ref{th-intro-discrete-fibers} has consequences for the existence of characteristic sections for maximal zip data (see below and \S\ref{subsec hodgetype}), and in particular for Hodge-type zip data (\Rmk~\ref{rmk_hodge_maximal}).

We say that a cocharacter datum $(G,\mu)$ is {\em maximal} if there exists a representation $r:G\to GL(V)$ with central kernel such that $\mu\circ r$ has exactly two weights. This implies that the parabolic subgroup of $GL(V)$ attached to $r \circ \mu$ is maximal. A zip datum $\Zcal$ is called maximal if $\Zcal=\Zcal_\mu$ for some maximal $(G,\mu)$. For example, a Shimura datum with good reduction at $p$ gives rise to a maximal cocharacter datum. We prove the following:

\begin{cor1}[\Cor~ \ref{corHT}]
If $(G,\mu)$ is maximal, then the stratification of $\GZip^\mu$ is uniformly principally pure.
\end{cor1}

This gives another proof of the existence of Hasse invariants for good reduction Shimura varieties of Hodge-type, first proved in \cite{Goldring-Koskivirta-Strata-Hasse}. In \loccit \Cor~3.1.3, we proved the special case of \Cor~ 1, when $(G,\mu)$ arises by reduction modulo $p$ from a Shimura datum of Hodge-type. The proof in \loccit used the methods of \cite{Deligne-Shimura-varieties} to show that the Hodge line bundle was orbitally $p$-close. By contrast with \cite{Goldring-Koskivirta-Strata-Hasse}, no characteristic zero methods are used here and no case-by-case examination is needed.

\subsection*{Flag spaces and affineness of strata}
In \cite{Goldring-Koskivirta-Strata-Hasse}, we generalized the construction of the flag space of \cite{EkedahlGeerEO} to arbitrary Hodge-type Shimura varieties. Here we generalize this construction further to arbitrary parabolics. 

Let $(\GG, X)$ be a Shimura datum of Hodge-type. Let $\Shk$ be the Kisin-Vasiu integral model of the associated Shimura variety $Sh_K(\GG,X)$ at a level $K$ which is hyperspecial at $p$. Denote the special fiber of $\Sscr_K$ by $S_K$. Recall that Zhang \cite{ZhangEOHodge} gives a smooth morphism $\zeta: S_K \to \GZip^\Zcal$, where $G$ is the special fiber of an extension of $\GG$ to a smooth, reductive group scheme over $\ZZ_{(p)}$.

Associated to a sub-parabolic $P_0 \subset P$, we define a partial flag space $\Fcal l_{K}$. It is a fibration over $\Shk$ with fiber $P/P_0$. The special fiber $F l_{K}$ of $\Fcal l_{K}$ fits in a cartesian diagram:

$$\xymatrix@1@M=5pt{
F l_{K} \ar[r]^-{\zeta_{P_0}} \ar[d]_-{\pi} & \GF^{(\Zcal, P_0)} \ar[d]^-{\pi_{P_0}} \\
S_K \ar[r]_-{\zeta} & \GZip^{\Zcal}
}$$
 Fine flag strata of $F l_K$ are defined as the preimages by $\zeta_{P_0}$ of those in $\GF^{(\Zcal, P_0)}$.

\begin{thmi}[\Th~\ref{thmshim}]  There exists an integer $N=N(\GG, X, K)$ such that, for all $p\geq N$, one has: 
\begin{enumerate}
\item $Sh_K(\GG,X)$ has good reduction at all places $\pfr|p$.
\item The fine stratification of $F l_{K}$ is uniformly principally pure.
\item If the closure $\overline{S}$ of a fine flag stratum $S\subset F l_{K}$  is proper, then $S$ is affine.
\end{enumerate}
\end{thmi}
The special case where the Shimura variety is of Siegel-type and $P_0=B$ was treated in \cite[\Prop 10.5~ii)]{EkedahlGeerEO}. Note that the hypothesis of (2) is satisfied if $S_K$ is proper.

\subsection*{Outline}
\S1 reviews the theory of $G$-zips, cocharacter data and zip data, and explains the relation between them. Maximal cocharacter data are also defined. Finally, we recall basic facts on the stack $\GZip^\Zcal$, and give a parametrization of the strata.

In \S2, we introduce the stack of $G$-zip flags, represent it as a quotient stack and define coarse flag strata.
\S3 is devoted to the definition of fine flag strata. The important notions of minimal and cominimal strata are defined.
\S4 studies sections of line bundles on fine flag strata and zip strata. In particular, \Th~\ref{thmonw} is proved and a counter-example to principal purity is given.
Some aspects of functoriality are investigated in \S5. The main results are \Th~\ref{discfib} on discreteness of fibers and \Cor~\ref{corHT} for maximal zip data.

We look at some applications to Shimura varieties  in \S6, where generalized flag spaces and their stratifications are defined. Finally, \Th~\ref{thmshim} is proved.

\section*{Acknowledgments}

The authors would like to thank Torsten Wedhorn for very stimulating discussions and advice. We are grateful to Beno\^{i}t Stroh for important ideas that influenced some aspects of the paper. Also, we thank Paul Ziegler for helpful discussions and correspondence. Finally we want to thank the reviewers for their suggestions on how to improve the paper.

\section{Review of $G$-Zips}
We fix a prime number $p$ and denote by $k$ an algebraic closure of $\FF_p$.

\subsection{The category of zip data}\label{catGzips}
Let $n\geq 1$ be an integer. In this paper, a \emph{zip datum} of exponent $n$ is a tuple $\Zcal=(G,P,L,Q,M,\varphi^n)$, where $G$ is a reductive group over $\FF_p$, $\varphi:G\to G$ the relative Frobenius morphism, $P,Q\subset G_k$ parabolics (not necessarily defined over $\FF_p$) with Levi subgroups $L\subset P$ and $M\subset Q$, and $n\geq 1$ an integer. We impose the condition $M=\varphi^n(L)$. Hence the $p^n$-Frobenius map restricts to $\varphi^n:L\to M$.

Zip data of exponent $n$ form a category $\Dscr_n$, where morphisms are defined as follows. Let $\Zcal_i=(G_,P_i,L_i,Q_i,M_i,\varphi^n)$ for $i=1,2$ two zip data of exponent $n$, and denote by $U_i\subset P_i$ and $V_i\subset Q_i$ the unipotent radicals. A morphism of zip data $\Zcal_1\to \Zcal_2$ is a group homomorphism $f:G_1\to G_2$ (defined over $\FF_p$) satisfying the conditions
\begin{equation}
 f(\square_1)\subset\square_2 \textrm{ for } \square=P,L,Q,M,U,V.
\end{equation}

For a zip datum $\Zcal=(G,P,L,Q,M,\varphi^n)$ and $g\in G(k)$, we define a conjugate zip datum
$${}^g \Zcal:=(G,{}^gP,{}^gL,{}^{\varphi^n(g)}Q,{}^{\varphi^n(g)}M,\varphi^n).$$

\subsection{Cocharacter data} \label{subsec cocharzipdata}
Let $K$ be a field with a fixed algebraic closure $\overline{K}$. We define the category of \emph{cocharacter data} over a field $K$ as the category of pairs $(G,\mu)$ where $G$ is a reductive group over $K$ and $\mu:\GG_{m,\overline{K}}\to G_{\overline{K}}$ is a cocharacter.

The category of cocharacter data is denoted by $\Dscr^{\rm co}$. Morphisms of cocharacter data $(G_1,\mu_1)\to (G_2,\mu_2)$ are group homomorphisms $f:G_1\to G_2$ (defined over $K$) satisfying $\mu_2=f\circ \mu_1$.

Let $(G,\mu)$ be a cocharacter datum. The cocharacter $\mu$ gives rise to a pair of opposite parabolics $(P_-, P_+)$ and a Levi subgroup $L:=P_-\cap P_+$. The Lie algebra of the parabolic $P_-$ (resp. $P_+$) is the sum of the non-positive (resp. non-negative) weight spaces of $\Ad\circ \mu$. More precisely, $P_+(\overline{K})$ consists of elements $g\in G(\overline{K})$ such that the limit
\begin{equation}\label{limparab}
\lim_{t\to 0}\mu(t)g\mu(t)^{-1}
\end{equation}
exists, i.e such that the map $\GG_{m,\overline{K}} \to G_{\overline{K}}$,  $t\mapsto\mu(t)g\mu(t)^{-1}$ extends to a morphism of varieties $\AA_{\overline{K}}^1\to G_{\overline{K}}$. The unipotent radical of $P_+$ is the set of such elements $g$ for which this limit is $1\in G(\overline{K})$.

When $K=\FF_p$, we define $P:=P_-$, $Q:=\varphi^n\left( P_+ \right)$ and $M:= \varphi^n(L)$. The tuple $\Zcal_\mu:=(G,P,L,Q,M,\varphi^n)$ is then a zip datum of exponent $n$. We call it the zip datum of exponent $n$ attached to the cocharacter datum $(G,\mu)$.

\begin{proposition}
The construction above gives rise to a faithful functor
\begin{equation}
\Zcal_n : \Dscr^{\rm co} \longrightarrow \Dscr_n, \ (G,\mu)\mapsto \Zcal_\mu
\end{equation}
\end{proposition}

\begin{proof}
This follows immediately from characterization (\ref{limparab}) above.
\end{proof}

\subsection{Frames of zip data} \label{subsec frames}
Let $\Zcal:=(G,P,L,Q,M,\varphi^n)$ be a zip datum of exponent $n\geq 1$. Let $(B,T)$ be a Borel pair in $G$ defined over $\FF_p$, which exists by Steinberg's Theorem.

Write $W:=W(G,T)$ for the Weyl group of $G_k$. Let $\Phi\subset X^*(T)$ (resp. $\Phi_L \subset X^*(T)$) be the set of $T$-roots of $G$ (resp. $L$). We define the set of positive roots $\Phi^+\subset \Phi$ by the condition that $\alpha \in \Phi^+$ when $U_{\alpha}\subset B$. Write $\Delta\subset \Phi^+$ for the set of positive simple roots. For $\alpha \in \Phi$, let $s_\alpha \in W$ be the corresponding reflection. Then $(W,\{s_\alpha, \alpha \in \Delta\})$ is a Coxeter group and we denote by $\ell :W\to \NN$ the length function.

Let $I\subset \Delta$ (resp. $J\subset \Delta$) be the type of $P$ (resp. $Q$). For any subset $K\subset \Delta$, let $W_K \subset W$ be the subgroup generated by $\{s_\alpha, \alpha \in K\}$.  Let $w_0$ (resp. $w_{0,K}$) be the longest element in $W$ (resp. $W_K$). Denote by ${}^K W$ (resp. $W^K$) the subset of elements $w\in W$ which are minimal in the coset $W_K w$ (resp. $wW_K$). The set ${}^K W$ (resp. $W^K$) is a set of representatives for the quotient $W_K \backslash W$ (resp. $W/W_K$).

\begin{definition} \label{defWframe}
Let $(B,T)$ be a Borel pair and $z\in W$. We call $(B,T,z)$ a \underline{$W$-frame} for $\Zcal$ if the following conditions are satisfied:
\begin{enumerate}[(i)]
\item $B,T$ are defined over $\fp$.
\item $B\subset P$.
\item ${}^z \!B\subset Q$.
\item \label{item-Wframe4} $ \varphi(B\cap L) =  {}^z \!B \cap M$.
\end{enumerate}
\end{definition}

\begin{rmk}\label{remBTcond} \
\begin{enumerate}[(1)]

\item If $(B,T)$ is defined over $\FF_p$, there exists $z\in W$ such that $(B,T,z)$ is a $W$-frame (proof of \cite[\Prop~3.7]{Pink-Wedhorn-Ziegler-zip-data}).
\item For any zip datum $\Zcal$, there exists $g\in G(k)$ such that ${}^g \Zcal$ admits a $W$-frame.
\item Similarly, for any cocharacter $\mu$ of $G$, there exists a $G(k)$-conjugate $\mu'$ of $\mu$ such that $\Zcal_{\mu'}$ admits a $W$-frame.
\item \label{item-cheval-syst} Let $(B,T,z)$ be a $W$-frame. For $w\in W$, choose a representative $\dot{w}\in N_G(T)$, such that $(w_1w_2)^\cdot = \dot{w}_1\dot{w}_2$ whenever $\ell(w_1 w_2)=\ell(w_1)+\ell(w_2)$ (this is possible by choosing a Chevalley system, \cite[XXIII, \S6]{SGA3}). Then $({}^z \! B,T,\dot{z}^{-1})$ is a frame as defined in \cite[\Def~3.6]{Pink-Wedhorn-Ziegler-zip-data}.
\item If $(B,T,z)$ is a $W$-frame, then condition \eqref{item-Wframe4} of \Def~\ref{defWframe} implies $z\in W^J$.
\end{enumerate}
\end{rmk}

\begin{definition} \label{defdualWframe}
Let $(B,T)$ be a Borel pair and $z\in W$. We call $(B,T,z)$ a \underline{dual $W$-frame} for $\Zcal$ if the following conditions are satisfied:
\begin{enumerate}[(i)]
\item $B,T$ are defined over $\fp$.
\item $B\subset Q$.
\item ${}^z \!B\subset P$.
\item $ \varphi({}^z \!B\cap L) =  B \cap M$.
\end{enumerate}
\end{definition}

The convention $B\subset P$ was used in \cite{Goldring-Koskivirta-Strata-Hasse}, contrary to \cite{Pink-Wedhorn-Ziegler-zip-data} and \cite{PinkWedhornZiegler-F-Zips-additional-structure}, who use the convention $B\subset Q$. The advantage of our definition of $W$-frame is that the parabolic $P$ exists naturally over the integral model of $G$, in the context of Shimura varieties. However, the result on Hasse invariants (\Th~\ref{sectionszipstrata}) can be stated more easily for a dual $W$-frame.

Let $(G,\mu)$ a cocharacter datum and denote by $\Zcal_\mu=(G,P,L,Q,M,\varphi^n)$ its associated zip datum of exponent $n$. Fix a Borel pair $(B,T)$ defined over $\FF_p$ such that $B\subset P$. We leave the following easy lemma to the reader:

\begin{lemma}\label{gtf}
The triple $(B,T,w_{0}w_{0,J})$ is a $W$-frame for the zip datum $\Zcal_\mu$.
\end{lemma}

\subsection{Maximal cocharacter data} \label{subsec hodgetype} Let $(G,\mu)$ be a cocharacter datum over $\FF_p$. We say that $(G,\mu)$ is maximal if there exists a finite-dimensional $\FF_p$-vector space $V$ and an $\FF_p$-representation with central (scheme-theoretic) kernel
\begin{equation}\label{maxtyper}
r:G\to GL(V)
\end{equation}
such that $r_k\circ \mu:\GG_{m,k}\to GL(V)_k$ has exactly two weights. Equivalently, there is a morphism of cocharacter data with central kernel $(G,\mu)\to (GL(V),\mu')$ such that $\mu'$ defines a maximal parabolic subgroup in $GL(V)_k$.

\begin{proposition}
Let $(G,\mu)$ be a cocharacter datum over $\FF_p$. The following assertions are equivalent:
\begin{enumerate}
\item \label{item-maxtype1} $(G,\mu)$ is maximal.
\item \label{item-maxtype2} There exists a finite-dimensional $k$-vector space $W$ and a representation $r':G_k\to GL(W)$ over $k$ with central kernel such that $r'\circ \mu$ has only two weights.
\end{enumerate}
\end{proposition}

\begin{proof}
Clearly \eqref{item-maxtype1} implies \eqref{item-maxtype2}. Conversely, let $r':G_k\to GL(W)$  be a representation as in \eqref{item-maxtype2}. Choosing a basis of $W$, we obtain a representation $r':G_k\to GL_{m,k}$. There exists $d\geq 1$ such that $r'$ is defined over $\FF_{q}$ where $q=p^d$, hence $r'$ comes from a morphism $r_1:G_{\FF_q}\to GL_{m,\FF_q}$. Then the representation
\begin{equation}
r:=\bigoplus_{\sigma \in \Gal(\FF_{q}/\FF_p)} {}^\sigma r_1
\end{equation}
is defined over $\FF_p$, so it comes from a morphism $r:G \to GL_{m',\FF_p}$ and it is clear that $r\circ \mu$ has only two weights.
\end{proof}

\begin{corollary}
Let $(G,\mu)$ be a maximal cocharacter datum over $\FF_p$. Let $H$ be a reductive group over $\FF_p$ and $\theta :G_k\to H_k$ an isomorphism. Then $(H,\theta \circ \mu)$ is maximal.
\end{corollary}

Let $(G,\mu)$ be maximal and let $r:G\to GL(V)$ over $\FF_p$ such that $\mu':=r_k \circ \mu$ has exactly two weights. Let $P\subset G_k$ denote the parabolic attached to $(G,\mu)$ by \S \ref{subsec cocharzipdata}. Similarly, write $P'_+,P'_-\subset GL(V)_k$ for the opposite parabolic subgroups attached to $\mu'$, as in \S \ref{subsec cocharzipdata}. Then $P'_-$ (resp. $P'_+$) is the stabilizer of a unique subspace $V_-\subset V_k$ (resp. $V_+\subset V_k$), and $V_k=V_-\oplus V_+$. Define a character $\chi'=\chi'(\mu, r) \in X^*(P'_-)$ by $\chi'(x)=\det(x|V_-)$.

\begin{definition}\label{charHTdef}
Let $(G,\mu)$ be a maximal cocharacter datum over $\FF_p$ and $\chi\in X^*(P)$. We say that $\chi$ is maximal if there exists $r:(G,\mu)\to (GL(V),\mu')$ as above such that $\chi=\chi'|_{P}$.
\end{definition}

\begin{rmk}
\label{rmk_hodge_maximal}
A Hodge-type zip datum (\cite[\Def~1.3.1]{Goldring-Koskivirta-Strata-Hasse}) is maximal.
\end{rmk}

\subsection{Stack of $G$-zips} \label{subsec gzips}
Let $\Zcal:=(G,P,Q,L,M,\varphi^n)$ be a zip datum of exponent $n$. Recall the following definition:

\begin{definition}[{\cite[\Def~3.1]{PinkWedhornZiegler-F-Zips-additional-structure}}]
A $G$-zip of type $\Zcal$ over a $k$-scheme $S$ is a tuple $\underline{\Ical}=(\Ical,\Ical_P,\Ical_Q,\iota)$ where $\Ical$ is a $G$-torsor over $S$, $\Ical_P\subset \Ical$ is a $P$-torsor over $S$, $\Ical_Q\subset \Ical$ is a $Q$-torsor over $S$, and $\iota: (\Ical_P)^{(p^n)}/R_u(P)^{(p^n)} \to \Ical_Q/R_u(Q)$ an isomorphism of $M$-torsors.
\end{definition}

The category of $G$-zips of type $\Zcal$ over $S$ is denoted by $\GZip^{\Zcal}(S)$. The $\GZip^{\Zcal}(S)$ give rise to a fibered category $\GZip^{\Zcal}$ over the category of $k$-schemes, which is a smooth stack of dimension $0$ (\cite[\Th~1.5]{PinkWedhornZiegler-F-Zips-additional-structure}).

We denote by $x\mapsto \overline{x}$ the natural projections $P\to L$ and $Q\to M$ given by reduction modulo the unipotent radical. The zip group is a subgroup of $P\times Q$ defined by:
\begin{equation}\label{zipgroup}
E_\Zcal:=\{(x,y)\in P\times Q, \varphi^n(\overline{x})=\overline{y}\}.
\end{equation}
The group $G\times G$ acts naturally on $G$ via $(a,b)\cdot g:= agb^{-1}$, and we let $P\times Q$ and $E$ act on $G$ by restricting this action. We can decompose $E_\Zcal$ as a semi-direct product:
\begin{equation}
(U  \times V) \rtimes L  \simeq   E_{\Zcal}, \ \ ((u,v),x)\longmapsto  (xu,\varphi(x)v). \\
\end{equation}
Recall the following theorem:
\begin{theorem}[\cite{PinkWedhornZiegler-F-Zips-additional-structure}, \Prop~3.11]
The stack of $G$-zips of type $\Zcal$ is isomorphic to the quotient stack $\left[E_\Zcal \backslash G\right]$.
\end{theorem}

\begin{lemma}\label{functo}
The above construction gives rise to a functor
\begin{equation}
\Re_n:\Dscr_n \longrightarrow (k-{\rm Stacks}), \ \Zcal \mapsto \GZip^\Zcal.
\end{equation}
\end{lemma}

\begin{proof}
Let $(G_1,P_1,L_1,Q_1,M_1,\varphi^n) \to (G_2,P_2,L_2,Q_2,M_2,\varphi^n)$ be a morphism of exponent $n$ zip data. Define
\begin{equation}
\alpha:G_1\times E_1 \to E_2, \ (g_1,p_1,q_1)\mapsto (f(p_1),f(q_1)).
\end{equation}
The fact that $f$ is a morphism of zip data ensures that $\alpha$ is well-defined. Conditions (1.1.1) and (1.1.2) of \S1.1 in \cite{Goldring-Koskivirta-Strata-Hasse} are satisfied, hence $(f,\alpha)$ induces a morphism of stacks $[E_1\backslash G_1] \to [E_2\backslash G_2]$.
\end{proof}

By composition, we obtain a family of functors $\Re_n \circ \Zcal_n$ from $\Dscr^{\rm co}$ to the 2-category of $k$-stacks. We call the stack attached in this way to the cocharacter datum $(G,\mu)$ the stack of $G$-zips of type $\mu$ (and exponent $n$), and we denote it by $\GZip_n^\mu$. For $n=1$, we denote it simply by $\GZip^\mu$.

\begin{proposition} \ \label{harmless}
\begin{enumerate}
\item Let $\mu$ be a cocharacter of $G_k$ and $\mu'=\adj(g)\circ \mu$ for some $g\in G(k)$. Then the $k$-stacks $\GZip_n^\mu$ and $\GZip_n^{\mu'}$ are isomorphic.
\item More generally, let $\Zcal=(G,P,L,Q,M,\varphi^n)$ be a zip datum, and $g\in G(k)$. Then the $k$-stacks $\GZip^\Zcal$ and $\GZip^{{}^g \Zcal}$ are isomorphic.
\end{enumerate}
\end{proposition}

\begin{proof}
It is easy to check that the map $G\to G$ given by $x\mapsto gx\varphi^n(g)^{-1}$ induces an isomorphism of stacks.
\end{proof}

\subsection{Stratification}\label{subsec stratif}
We recall the results of \cite{Pink-Wedhorn-Ziegler-zip-data}, but reformulated using our convention (\Def~\ref{defWframe}). Choose a system of representatives $\dot{w}\in G$ for $w\in W$ as in \Rmk~\ref{remBTcond}\eqref{item-cheval-syst}.

Let $\Zcal=(G,P,L,Q,M,\varphi^n)$ be a zip datum, and let $(B,T,z)$ be a $W$-frame of $\Zcal$. If $h\in G(k)$ is an element of the group, we denote by $\Ocal_\Zcal(h)$ the $E$-orbit of $h$ in $G$. Similarly, we write $o_\Zcal(h):=\left[ E_\Zcal \backslash \Ocal_\Zcal(h)\right]$ for the corresponding $k$-point in $\left[E_\Zcal \backslash G\right]$. For $w\in W$, define $G_w:=\Ocal_\Zcal(\dot{w} \dot{z}^{-1})$.

By \Ths~7.5, 11.3 in \loccitn, the map $w\mapsto G_w$ restricts to two bijections
\begin{align} 
{}^I W \to \{E_\Zcal \textrm{-orbits in }G\}\label{param}\\
W^J \to \{E_\Zcal \textrm{-orbits in }G\}\label{paramdual}
\end{align}

Furthermore, one has the following dimension formulas:
\begin{equation}\label{dimzipstrata}
\dim(G_w)= \ell(w)+\dim(P)\quad \textrm{for all }w\in {}^I W \cup W^J.
\end{equation}

\section{The stack of zip flags}\label{sec stofzf}

Let $\Zcal=(G,P,L,Q,M,\varphi^n)$ be a zip datum and set $q:=p^n$. We assume that $\Zcal$ admits a W-frame $(B,T,z)$, which we fix (this assumption is harmless by \Rmk~\ref{remBTcond} (2) and \Prop~ \ref{harmless}). As in subsection \ref{subsec frames}, we denote by $\Phi,\Phi^+,\Delta$ the roots, positive roots, simple roots with respect to $(B,T)$ respectively, and we denote by $I,J\subset \Delta$ the types of $P,Q$ respectively.

\subsection{Definition of the stack $\GF^{(\Zcal,P_0)}$}
We denote by $P_0$ a parabolic subgroup such that $B \subset P_0 \subset P$, and we let $I_0\subset I$ denote the type of $P_0$. We call $(\Zcal,P_0)$ a flagged zip datum.

\begin{definition}
A $G$-zip flag of type $(\Zcal,P_0)$ and exponent $n$ over a $k$-scheme $S$ is a pair $\hat{\Ical}=(\underline{\Ical},\Jcal)$ where $\underline{\Ical}=(\Ical,\Ical_P,\Ical_Q,\iota)$ is a $G$-zip of type $\Zcal$ and exponent $n$ over $S$, and $\Jcal\subset \Ical_P$ is a $P_0$-torsor.
\end{definition}
We denote by $\GF^{(\Zcal,P_0)}(S)$ the category of $G$-zip flags over $S$ of type $(\Zcal,P_0)$. By the same arguments as for $G$-zips, we obtain a stack $\GF^{(\Zcal,P_0)}$ over $k$, which we call the stack of $G$-zip flags of type $(\Zcal,P_0)$. There is a natural morphism of stacks
\begin{equation}
\pi_{P_0}:\GF^{(\Zcal,P_0)}\longrightarrow \GZip^\Zcal
\end{equation}
which forgets the $P_0$-torsor $J$ on the level of categories $\GF^{(\Zcal,P_0)}(S) \to \GZip^\Zcal(S)$. More generally, for each pair of parabolics $(P_0,P_1)$ such that $B\subset P_1\subset P_0\subset P$, we have a natural map of stacks
\begin{equation}\label{pimap}
\pi_{P_1,P_0}:\GF^{(\Zcal,P_1)}\longrightarrow \GF^{(\Zcal,P_0)}.
\end{equation}
Note that when $P_0=P$, the stack $\GF^{(\Zcal,P)}$ identifies with $\GZip^\Zcal$ through the isomorphism $\pi_P$. Hence we can identify the maps $\pi_{P_0}$ and $\pi_{P_0,P}$. We obtain a tower of stacks above $\GZip^\Zcal$. Define a subgroup $\hat{E}_{P_0}\subset E_\Zcal$ by:
\begin{equation}
\hat{E}_{P_0}:=\{(x,y)\in E_\Zcal, \ x\in P_0\}.
\end{equation}

To avoid confusion, we will also denote the group $\hat{E}_{P_0}$ by $\hat{E}^{\Zcal}_{P_0}$ to emphasize that it depends on the zip datum $\Zcal$.

\begin{theorem}\label{comdiagGF} \
\begin{enumerate}
\item The stack $\GF^{(\Zcal,P_0)}$ is a smooth stack of dimension $\dim(P/P_0)$, isomorphic to $\left[\hat{E}_{P_0} \backslash G\right]$.
\item \label{item-comdiag} For all parabolics $B\subset P_1\subset P_0\subset P$, we have a commutative diagram
$$\xymatrix@1@M=5pt@R=1cm@C=1.5cm{
\GF^{(\Zcal,P_1)} \ar[r]^{\pi_{P_1,P_0}} \ar[d]^\simeq & \GF^{(\Zcal,P_0)} \ar[d]^{\simeq} \\
\left[\hat{E}_{P_1} \backslash G\right] \ar[r] & \left[\hat{E}_{P_0} \backslash G\right]
}$$
where the vertical maps are isomorphisms and the lower horizontal map is induced by the natural projection.
\item \label{item-fibers} The map $\pi_{P_1,P_0}$ is proper and smooth, with fibers isomorphic to $P_0/P_1$.
\end{enumerate}
\end{theorem}
\begin{proof}
The proof is entirely similar to the proof of \cite[\Th~5.1.3]{Goldring-Koskivirta-Strata-Hasse}. The assertion on the dimension follows  from \eqref{item-fibers} and $\dim(\GZip^\Zcal)=0$.
\end{proof}

In other words, \eqref{item-comdiag} states that there is an isomorphism of towers $\GF^{(\Zcal,P_\bullet)} \simeq \left[\hat{E}_{P_\bullet} \backslash G\right]$. We set $P_{0,L}:= P_0\cap L$, and similarly $B_L:=B\cap L$. The inclusion $G\to G\times L$, $g\mapsto (g,1)$ induces an isomorphism of quotient stacks
\begin{equation}\label{isomGF}
\left[\hat{E}_{P_0} \backslash G\right] \longrightarrow \left[E \backslash \left( G\times L/P_{0,L} \right)\right].
\end{equation}
In this description, the same group $E$ acts at all levels of the tower, and the natural projections are $E$-equivariant.

\subsection{Coarse flag strata}\label{subsec coarse}
We identify without explicit mention the stack of zip flags of type $(\Zcal,P_0)$ and the quotient stack $\left[\hat{E}_{P_0} \backslash G\right]$ (\Th~\ref{comdiagGF}). Define a parabolic subgroup $Q_0$ by
\begin{equation}\label{Q0}
Q_{0}:=(L\cap P_0)^{(p^n)}R_u(Q)
\end{equation}
Note that ${}^{z}\!B\subset Q_0\subset Q$ and $\hat{E}_{P_0} \subset P_0\times Q_0$. Denote by $J_0\subset J$ the type of $Q_0$. Hence we have a natural smooth surjective morphism of stacks:
\begin{equation}\label{coarsedef}
\psi_{P_0} : \GF^{(\Zcal,P_0)}=\left[\hat{E}_{P_0} \backslash G\right] \to \left[(P_0\times Q_0) \backslash G \right].
\end{equation}
We define the coarse flag strata of the stack $\GF^{(\Zcal,P_0)}$ as the fibers of the morphism $\psi_{P_0}$ (endowed with the reduced structure). Hence the coarse flag strata are the quotient stacks $[\hat{E}_{P_0}\backslash C]$ for $C$ a $P_0\times Q_0$-orbit in $G$. Define a closed coarse flag stratum as the preimage of the closure of a point in $\left[(P_0\times Q_0) \backslash G \right]$ by $\psi_{P_0}$ (endowed with the reduced structure). Using the isomorphism (\ref{isomGF}), we obtain a map
\begin{equation}
\left[E \backslash \left( G\times L/P_{0,L} \right)\right] \to \left[(P_0\times Q_0) \backslash G \right].
\end{equation}
which is induced by the map: $\tilde{\psi}_{P_0}: G\times L \to G, \ \ (g,x)\mapsto x^{-1}g\varphi^n(x)$.

Recall that the underlying topological space of $\left[(P_0\times Q_0) \backslash G \right]$ is isomorphic to $^{I_0}W^{J_0}$ endowed with the topology attached to the Bruhat order. The parametrization is given by the map:
\begin{equation}
^{I_0}W^{J_0} \to \left[(P_0\times Q_0) \backslash G \right], \ \ w\mapsto P_0\dot{w}\dot{z}^{-1}Q_0.
\end{equation}
For $w\in {}^{I_0}W^{J_0}$, we denote by $C_{P_0,w}:=P_0\dot{w}\dot{z}^{-1}Q_0$ the corresponding $P_0\times Q_0$-orbit, and define $$\ZZ_{P_0,w}:=\left[ \hat{E}_{P_0} \backslash C_{P_0,w} \right]$$
the corresponding coarse flag stratum. Similarly, using the description given by isomorphism (\ref{isomGF}), we also define:
\begin{equation}\label{coarsevar}
\HH_{P_0,w}:=\{(g,x)\in G\times L/P_{0,L}, \ x^{-1}g\varphi^n(x)\in C_{P_0,w} \}.
\end{equation}
Thus $\HH_{P_0,w}$ is a locally closed $E$-stable subvariety of $G\times L/P_{0,L}$ and $\ZZ_{P_0,w}=[E \backslash \HH_{P_0,w}]$.

\begin{proposition}\ \label{coarseprop}
\begin{enumerate}
\item The coarse flag strata are irreducible, smooth, and their closures are normal.
\item The coarse closed flag strata coincide with the closures of coarse flag strata.
\item Set $I_w:=J_0\cap w^{-1}I_0 w$. One has $\dim(\HH_{P_0,w}) =\ell(w)+\ell(w_{0,J_0})-\ell(w_{0,I_w})-\dim(P_0)$.
\end{enumerate}
\end{proposition}

\begin{proof}
The first two parts are a simple generalization of \Th~5.3.3 in \cite{Goldring-Koskivirta-Strata-Hasse}. The third one is a consequence of \Prop~1.12 in \cite{Wedhorn-bruhat}.
\end{proof}

For parabolics $B\subset P_1\subset P_0$ with corresponding subgroups ${}^{z}\!B\subset Q_1\subset Q_0$, we have a commutative diagram
$$\xymatrix@1@M=5pt@R=1cm@C=2cm{
\GF^{(\Zcal,P_1)} \ar[d]^{\pi_{P_1,P_0}} \ar[r]^{\psi_{P_1}} & \left[(P_1\times Q_1) \backslash G \right]\ar[d]\\
\GF^{(\Zcal,P_0)} \ar[r]^{\psi_{P_0}} & \left[(P_0\times Q_0) \backslash G \right]
}$$
where the rightmost vertical map is the natural projection. Hence the maps $(\psi_{P_0})_{P_0}$ define a morphism of towers
\begin{equation}
\psi:\GF^{(\Zcal,P_\bullet)}\to \left[(P_\bullet \times Q_\bullet ) \backslash G \right].
\end{equation}

\section{Fine flag strata}\label{sec ffs}

In this section, we define a stratification of $\GF^{(\Zcal,P_0)}$ which is finer that the stratification given by $\psi_{P_0}$.

\subsection{Definition of fine flag strata} \label{subsec def fine strata}
For a parabolic  $B\subset P_0\subset P$ with type $I_0\subset I$, we always denote by $Q_0$ the parabolic attached to $P_0$ by formula (\ref{Q0}) and by $J_0$ its type. Let $L_0$ (resp. $M_0$) be the Levi subgroup of $P_0$ (resp. $Q_0$) containing $T$. It follows from the definition that $\varphi^n(L_0)=M_0$. Hence the tuple
\begin{equation}\label{zipparab}
\Zcal_{0}:=(G,P_0,L_0,Q_0,M_0,\varphi^n)
\end{equation}
is a zip datum of exponent $n$. We obtain a stack $\GZip^{\Zcal_{0}}$ and a zip group $E_{\Zcal_{0}}$ such that $\hat{E}_{P_0}\subset E_{\Zcal_{0}} \subset P_0\times Q_0$. Hence we obtain a smooth surjective morphism of stacks
\begin{equation}\label{Psimap}
\Psi_{P_0}:\GF^{(\Zcal,P_0)}\to \GZip^{\Zcal_{0}}\simeq \left[E_{\Zcal_0} \backslash G \right]
\end{equation}
which induces a factorization of $\psi_{P_0}$:
\begin{equation}
\GF^{(\Zcal,P_0)} \xrightarrow{\Psi_{P_0}} \GZip^{\Zcal_{0}}\simeq \left[\hat{E}_{P_0} \backslash G \right] \xrightarrow{\beta} \left[(P_0\times Q_0) \backslash G \right],
\end{equation}
where $\beta$ is the natural projection. Define fine flag strata as the fibers of the map $\Psi_{P_0}$ (endowed with the reduced structure). Fine flag strata coincide with coarse flag strata when $P_0=B$ and with usual zip strata when $P_0=P$.

We can decompose the groups $E_{\Zcal_0}$ and $\hat{E}_{P_0}$ as semi-direct products:
\begin{align}\label{decompE1}
& (U_0  \times V_0) \rtimes L_0 \simeq   E_{\Zcal_0}, \ \ ((u,v),x)\longmapsto  (xu,\varphi(x)v) \\
&(U_0  \ltimes V) \rtimes L_0   \simeq   \hat{E}_{P_0}, \ \ ((u,v),x) \longmapsto  (xu,\varphi(x\theta^{P}_{L}(u))v) \label{decompE2}
\end{align}
where $\theta^{P}_L : P \to L$ denotes the reduction modulo the unipotent radical $U$.

\begin{lemma}\label{affinebun}
One has an isomorphism $E_{\Zcal_0}/\hat{E}_{P_0}\simeq M \cap V_0 \simeq V_0/V \simeq \AA^{\dim(P/P_0)}$.
\end{lemma}

\begin{proof}
Using \eqref{decompE1} and \eqref{decompE2}, one sees that the map $
V_0 \to E_{\Zcal_0}/\hat{E}_{P_0}$ taking $u\in V_0$ to the class of $(1,u)\in E_{\Zcal_0}$ is surjective, and induces an isomorphism $V_0/V \simeq E_{\Zcal_0}/\hat{E}_{P_0}$.
It is also clear that the inclusion $M \cap V_0 \subset V_0$ induces an isomorphism $M \cap V_0 \simeq V_0/V$. Since $M\cap V_0$ is unipotent, it is isomorphic as a variety to $\AA^r$, with $r=\dim(M\cap V_0)=\dim(P/P_0)$.
\end{proof}

For parabolic subgroups $B\subset P_1 \subset P_0 \subset P$, one has inclusions
$$\xymatrix@=5pt@M=2pt@R=2pt{
\hat{E}_{P_1} & \subset & \hat{E}_{P_0} \\
\cap && \cap \\
E_{\Zcal_1} & & E_{\Zcal_0}}$$
but in general there is no inclusion between $E_{\Zcal_0}$ and $E_{\Zcal_1}$.

\begin{proposition}
Let $B\subset P_1 \subset P_0 \subset P$ be parabolic subgroups. Then the map $\pi_{P_1,P_0}$ sends a fine flag stratum for $P_1$ to a union of fine flag strata for $P_0$.
\end{proposition}

\begin{proof}
Denote by $\Zcal_1$ the zip datum attached to $P_1$ as in (\ref{zipparab}). If $C$ is an $E_{\Zcal_1}$-orbit in $G$, we must show that the set $X:=\hat{E}_{P_0}\cdot C$ is $E_{\Zcal_0}$-stable. Since $X$ is clearly stable by $\varphi^n$-conjugation by $L_0$, it suffices to show that $X$ is stable by elements $(u,v)\in R_uP_0\times R_uQ_0$. Let $x\in X$ be an element. Since $R_uP\times R_uQ\subset \hat{E}_{P_0}$, we may assume that there exists $a\in P_{0,L}$ such that $y:=ax\varphi^n(a)^{-1}$ is in $C$. Write $a=a'r$ for $a'\in L_0$ and $r\in R_uP_0$. Then
$$uxv=ua^{-1}y\varphi^n(a)v=a'^{-1}u'yv'\varphi^n(a')$$
for some elements $(u',v')\in R_uP_0\times R_uQ_0$. Since $R_uP_0\times R_uQ_0 \subset R_uP_1\times R_uQ_1$, we deduce that $u'yv'\in C$, and hence $uxv\in X$ as claimed.
\end{proof}

For a parabolic  $B\subset P_0\subset P$ with type $I_0\subset I$ and $w\in W$, define
\begin{equation}
Z_{P_0,w}:=\Psi_{P_0}^{-1}\left( o_{\Zcal_0}(\dot{w}\dot{z}^{-1})  \right)
\end{equation}
endowed with the reduced structure as locally closed substack. Similarly to the parametrizations \eqref{param} and \eqref{paramdual} for $E$-orbits in $G$, the map $w\mapsto Z_{P_0,w}$ restricts to two bijections
\begin{align}
{}^{I_0}W \longrightarrow &\{\textrm{fine flag strata of }\GFZP\} \label{paramffs} \\
W^{J_0} \longrightarrow &\{\textrm{fine flag strata of }\GFZP\} \label{paramffsdual}.
\end{align}
For $w\in {}^{I_0}W\cup W^{J_0}$, define a locally closed subvariety of $G\times (L/P_{0,L})$ by
\begin{equation}\label{finevar}
H_{P_0,w}:=\{(g,x)\in G\times L/P_{0,L}, \ x^{-1}g\varphi^n(x)\in \Ocal_{\Zcal_0}(\dot{w}\dot{z}^{-1}) \}.
\end{equation}
It is $E$-stable and one has $Z_{P_0,w}=\left[E\backslash H_{P_0,w} \right]$.

\begin{proposition} 
For $w\in {}^{I_0}W\cup W^{J_0}$, the variety $H_{P_0,w}$ (resp. the stack $Z_{P_0,w}$) is smooth of dimension $\ell(w)+\dim(P)$ (resp. $\ell(w)+\dim(P)-\dim(G)$).
\end{proposition}

\begin{proof}
This follows immediately from \Prop~\ref{affinebun} and the dimension formula (\ref{dimzipstrata}).
\end{proof}

If different zip data come into play, we write a superscript indicating the zip datum one considers. For example, consider parabolics $B\subset P_1 \subset P_0 \subset P$. We have defined a zip datum $\Zcal_0=(G,P_0,L_0,Q_0,M_0,\varphi^n)$. Hence it makes sense to speak of $G$-zip flags of type $(\Zcal_0,P_1)$. Therefore we define $\pi^{\Zcal_0}_{P_1}$ and $\Psi^{\Zcal_0}_{P_1}$ as the maps attached to $P_1$ with respect to the zip datum $\Zcal_0$ by (\ref{pimap}) and (\ref{Psimap}) respectively.

\begin{proposition}\label{trick}
Let $B\subset P_1 \subset P_0 \subset P$ be parabolic subgroups. There is a natural map $$\Psi^{\Zcal/P_0}_{P_1}:\GF^{(\Zcal,P_1)}\to \GF^{(\Zcal_0,P_1)}$$ such that the following diagram commutes:
$$\xymatrix@1@M=8pt@C=10pt{
& \GZip^{\Zcal_1}  & \\
\GF^{(\Zcal,P_1)}\ar[rr]^{\Psi^{\Zcal/P_0}_{P_1}} \ar[ru]^-{\Psi^{\Zcal}_{P_1}} \ar[d]^-{\pi^{\Zcal}_{P_1,P_0}} && \GF^{(\Zcal_0,P_1)} \ar[ul]_-{\Psi^{\Zcal_0}_{P_1}} \ar[d]^-{\pi^{\Zcal_0}_{P_1}}\\
\GF^{(\Zcal,P_0)} \ar[rr]^-{\Psi^{\Zcal}_{P_0}} && \GZip^{\Zcal_0}
}$$
Moreover, the commutative square above is cartesian.
\end{proposition}

\begin{proof}
This follows simply from the bicartesian square
$$\begin{tikzcd}
\hat{E}^{\Zcal}_{P_1} \arrow[hookrightarrow]{d} \arrow[hookrightarrow]{r} & \hat{E}^{\Zcal_0}_{P_1} \arrow[hookrightarrow]{d} \\
\hat{E}^{\Zcal}_{P_0} \arrow[hookrightarrow]{r} & E_{\Zcal_0}
\end{tikzcd}$$
We mean by this that the subgroups $\hat{E}^{\Zcal_0}_{P_1}$ and $\hat{E}^{\Zcal}_{P_0}$ generate $E_{\Zcal_0}$ and their intersection is  exactly $\hat{E}^{\Zcal}_{P_1}$.
\end{proof}
\begin{rmk}
The morphism $\Psi^{\Zcal/P_0}_{P_1}$ induces thus a bijection between fine flag strata of $\GF^{(\Zcal,P_1)}$ and $\GF^{(\Zcal_0,P_1)}$. Hence, any statement stable by base change regarding intermediate parabolics $P_1\subset P_0$ can be reduced to a result involving just the parabolic $P_1$, by changing the zip datum to $\Zcal_0$.
\end{rmk}



\Prop~\ref{trick} can be generalized as follows. The proof is similar, so we skip it. Consider four parabolics $B\subset P_\square\subset P_\blacksquare \subset P_\blacklozenge \subset P_\lozenge \subset P$. There is a commutative diagram with a cartesian square in the middle:
$$\xymatrix@1@M=8pt@C=10pt{
& \GZip^{\Zcal_\square}  & \\
\GF^{(\Zcal_\lozenge,P_\square)} \ar[ru]^-{\Psi^{\Zcal_\lozenge}_{P_\square}} \ar[rr]^{\Psi^{\Zcal_\lozenge/P_\blacklozenge}_{P_\square}} \ar[d]^-{\pi^{\Zcal_\lozenge}_{P_\square,P_\blacksquare}} && \GF^{(\Zcal_\blacklozenge,P_\square)}  \ar[d]^-{\pi^{\Zcal_\blacklozenge}_{P_\square,P_\blacksquare}}  \ar[ul]_-{\Psi^{\Zcal_\blacklozenge}_{P_\square}}\\
\GF^{(\Zcal_\lozenge,P_\blacksquare)} \ar[rr]^-{\Psi^{\Zcal_\lozenge/P_\blacklozenge}_{P_\blacksquare}} \ar[rd]_-{\Psi^{\Zcal_\lozenge}_{P_\blacksquare}} && \GF^{(\Zcal_\blacklozenge,P_\blacksquare)} \ar[dl]^-{\Psi^{\Zcal_\blacklozenge}_{P_\blacksquare}}\\
& \GZip^{\Zcal_\blacksquare}  &
 }$$
\subsection{Minimal, cominimal fine flag strata}
In general, it is difficult to determine the image of a fine flag stratum by $\pi_{P_1,P_0}$. In this section we give a partial result.

\begin{definition}
Let  $B\subset P_1\subset P_0\subset P$ be parabolics of type $I_1\subset I_0\subset I$. A $P_0$-minimal (resp. $P_0$-cominimal) fine flag stratum of $\GF^{(\Zcal,P_1)}$ is a stratum of the form $Z_{P_1,w}$ for $w\in {}^{I_0} W$ (resp.  $w\in W^{J_0}$).
\end{definition}

This generalizes the concept of minimal strata introduced by Ekedahl and Van der Geer in \cite{EkedahlGeerEO} and also studied by the authors in \cite{Goldring-Koskivirta-Strata-Hasse}.

\begin{proposition} \label{propimagemin} Let $B\subset P_1\subset P_0\subset P$ be parabolics of type $I_1\subset I_0\subset I$. For all $w\in {}^{I_0}W\cup W^{J_0}$, one has:
\begin{equation}
\pi_{P_1,P_0}\left(Z_{P_1,w}\right)=Z_{P_0,w}.
\end{equation}
Moreover, the map $\pi_{P_1,P_0}: Z_{P_1,w} \to Z_{P_0,w}$ is finite.
\end{proposition}

\begin{proof}
In the special case $P_1=B$ and $P_0=P$, the result follows from \cite[\Prop~5.4.3]{Goldring-Koskivirta-Strata-Hasse}. Using \Prop~\ref{trick}, we obtain the result for arbitrary $P_0$ and $P_1=B$. Now, let $P_1\subset P_0$ be arbitrary. By the special case, we have in particular $\pi_{B,P_1}(Z_{B,w})=Z_{P_1,w}$. Hence:
\begin{equation}
\pi_{P_1,P_0}\left(Z_{P_1,w}\right)=\pi_{P_1,P_0}\left(\pi_{B,P_1}\left(Z_{B,w}\right)\right)=\pi_{B,P_0}\left(Z_{B,w}\right)=Z_{P_0,w}
\end{equation}
where the last equality follows again from the special case.

For the finiteness of the map $\pi_{P_1,P_0}: Z_{P_1,w} \to Z_{P_0,w}$, we reduce again to the case $P_1=B$, where it is known by \cite[\Prop~5.4.3]{Goldring-Koskivirta-Strata-Hasse}. We use the varieties $H_{P_0,w}$ defined in (\ref{finevar}) to use classical algebraic geometry. Notice that $\pi_{P_1,P_0}$ induces a surjective proper map $\tilde{\pi}_{P_1,P_0}:H_{P_1,w} \to H_{P_0,w}$, and it suffices to show that this map is quasi-finite. The map $\tilde{\pi}_{B,P_0}$ factors as a composition $H_{B,w} \xrightarrow{\tilde{\pi}_{B,P_1}} H_{P_1,w} \xrightarrow{\tilde{\pi}_{P_1,P_0}} H_{P_0,w}$. Since $\tilde{\pi}_{B,P_1}$ is surjective and the composition is quasi-finite, the second map must be quasi-finite. This concludes the proof.
\end{proof}

\subsection{Closure relations for $\GZip^\Zcal$}
For $w\in {}^IW$, write $G^w:=\Ocal(\dot{w}\dot{z}^{-1})$ as in \S\ref{subsec stratif}. Denote by $\leq$ the Bruhat order on $W$. For $w,w'\in {}^IW$, define $w\preceq w'$ if and only if there exists $y\in W_I$ such that $yw'\varphi^n(y)^{-1} \leq w$. Recall the following:
\begin{theorem}[{\cite[\Th~6.2]{Pink-Wedhorn-Ziegler-zip-data}}]\label{thmPWZorder}
The Zariski closure of $G^w$ is the union of all $G^{w'}$ with $w'\preceq w$.
\end{theorem}

We show below how the flag space can be used to prove a similar result. Consider the flag space $\GF^{(\Zcal,B)}$, which we denote simply by $\GF^\Zcal$. Since we only consider the parabolic $P$, we shorten the term "$P$-minimal" to "minimal" and similarly for "cominimal". For elements $x,y\in G(k)$, we define two relations:
\begin{align}
x\leq y &\Longleftrightarrow x\in\overline{B y {}^z\!B} \label{Border}\\
x\preceq y &\Longleftrightarrow x \in \overline{\Ocal_\Zcal(y)} \label{curlyorder}
\end{align}
We say that $(g,l)\in G\times L$ is minimal (resp. cominimal) if its image by the isomorphism (\ref{isomGF}) in $\GF^\Zcal$ lies in a minimal (resp. cominimal) stratum. We say that $g\in G$ is minimal (resp. cominimal) if $(g,1)$ is minimal (resp. cominimal).

\begin{proposition} \label{relequiv}
Let $x,y\in G$, with $y$ minimal or cominimal. Then:
\begin{equation}\label{equivtwisted}
x\preceq y \ \Longleftrightarrow \ \exists h\in L, \ hx\varphi^n(h)^{-1}\leq y.
\end{equation}
\end{proposition}

\begin{proof}
We will assume that $y$ is minimal, since the proof is entirely similar in the cominimal case. Let $H\subset G\times (L/B_L) $ denote the minimal stratum containing the image of $(y,1)$. By \Prop~\ref{propimagemin}, the image $\pi(H)$ by the projection $\pi :G\times (L/B_L) \to G$ is $\Ocal(y)$. Now, the right-hand side of \eqref{equivtwisted} amounts to the existence of $h\in L$ such that $(x,\bar{h})\in\overline{H}$, where $\bar{h}$ is the image of $h$ by the map $L\to L/B_L$. But this means exactly that $x$ belongs to $\pi(\overline{H})=\overline{\Ocal(y)}$, so it is equivalent to $x \preceq y$.
\end{proof}

\begin{rmk}
For $w\in {}^I W$, the element $\dot{w}\dot{z}^{-1}$ is minimal. Applying \Prop~\ref{relequiv} to this element gives a similar result as \cite[\Th~6.2]{Pink-Wedhorn-Ziegler-zip-data} recalled above.
\end{rmk}

\section{Sections over fine flag strata}

\subsection{Existence of sections on zip strata}
We recall some results of \cite[\S3]{Goldring-Koskivirta-Strata-Hasse} on the existence of group-theoretical Hasse invariants. The formulas generalize readily to an arbitrary exponent $n\geq 1$.

\begin{definition}\label{charsecdef}
Let $S$ be locally closed in a scheme (or algebraic stack) $X$, with Zariski closure $\overline{S}$ (endowed with the reduced structure), and let $\Lscr$ be a line bundle on $X$. We say that a section $f\in H^0(\overline{S},\Lscr)$ is a characteristic section for $S$ if the non-vanishing locus of $f$ in $\overline{S}$ is exactly $S$.
\end{definition}

Fix a zip datum $\Zcal$ of exponent $n$. Assume that there exists a torus $T\subset L$ defined over $\FF_p$ (we may achieve this by replacing $\Zcal$ with ${}^g\Zcal$ for an appropriate $g\in G$, see \Prop~\ref{harmless}). Identify $X^*(E)=X^*(L)$ via the maps $E\to P\to L$ (where the first map is the first projection, and the second map is the natural projection). For $\lambda \in X^*(L)$, denote by $\Vscr(\lambda)$ the line bundle on $\left[E\backslash G \right]\simeq\GZip^\Zcal$ naturally attached to $\lambda$ by \cite[\S N.3.1]{Goldring-Koskivirta-Strata-Hasse}.

\begin{definition} \ 
\label{def-orb-p-close}
\begin{enumerate}
\item Let $\chi \in X^*(T)$. We say that $\chi$ is \underline{orbitally $q$-close} if $\left| \frac{\langle \chi, \sigma \alpha^{\vee} \rangle}{\langle \chi, \alpha^{\vee} \rangle} \right| \leq q-1$ for every $\alpha \in \Phi$ satisfying $\langle \chi , \alpha^{\vee} \rangle \neq 0$ and for all $\sigma \in   W \rtimes \Gal(k/\FF_p)$.
\item We say that $\chi$ is \underline{$q$-small} if $|\langle \chi, \alpha^{\vee} \rangle| \leq q-1$ for all $\alpha \in \Phi$.
\end{enumerate}
\end{definition}

\begin{definition} \label{defample}
A character $\lambda\in X^*(L)$ is $\Zcal$-\underline{ample} if the associated line bundle on $G/\varphi^{-n}(Q)$ is anti-ample (note that $L\subset \varphi^{-n}(Q)$ is a Levi subgroup).
 \label{def-ample-character} \end{definition}
\begin{rmk}\label{rmk-ample} Let $(B,T,z)$ be a $W$-frame and $(B',T',z')$ a dual $W$-frame. The following are equivalent:
\begin{enumerate}
\item $\chi$ is $\Zcal$-ample.
\item One has $\langle \chi,\varphi^{-n}(z)\alpha^\vee \rangle <0$ for all $\alpha\in \Delta \setminus \varphi^{-n}(J)$ with respect to $(B,T)$.
\item One has $\langle \chi,\alpha^\vee \rangle <0$ for all $\alpha\in \Delta \setminus \varphi^{-n}(J)$ with respect to $(B',T')$.
\end{enumerate} 
\end{rmk}

\begin{rmk}\label{rmk-ample-case} Let $\mu \in X_*(G)$ and $\Zcal=\Zcal_\mu$. Let $(B,T,z)$ be a $W$-frame, $B\subset P_0 \subset P$ a parabolic and $\Zcal_0$ defined by \eqref{zipparab}. The following are equivalent:
\begin{enumerate}
\item $\chi$ is $\Zcal_0$-ample.
\item 
\label{item-GS}
One has $\langle \chi,\alpha^\vee \rangle >0$ for all $\alpha\in I \setminus I_0$ and $\langle \chi,\alpha^\vee \rangle <0$ for all $\alpha\in \Phi^+ \setminus \Phi_L^+$. 
\end{enumerate}
\end{rmk}

\begin{rmk} 
\label{rmk-GS}
Condition~\ref{rmk-ample-case}\eqref{item-GS} is precisely the one that arises in the work of Griffiths-Schmid concerning the positivity of automorphic vector bundles on homogeneous complex manifolds, see the proof following $(4.21)_D$ on p. 276 of \cite{Griffiths-Schmid-homogeneous-complex-manifolds}. We will elaborate on the connection between \cite{Griffiths-Schmid-homogeneous-complex-manifolds} and sections on Zip strata in our forthcoming work with Brunebarbe and Stroh \cite{Brunebarbe-Goldring-Koskivirta-Stroh-ampleness}.

\end{rmk}

Fix $N\geq 1$ such that $\dim_k H^0(\left[ E\backslash S\right],\Vscr(N\chi))= 1$ for all $E$-orbits $S\subset G$ and for all $\chi \in X^*(L)$ (for the existence of such $N$, see \loccit \S3). Denote by $h_{S,\chi}$ a nonzero element of $H^0(\left[ E\backslash S\right],\Vscr(N\chi))$. Let $\overline{S}$ be the Zariski closure of $S$, endowed with the reduced structure.

\begin{theorem}[{\cite[\Th~3.2.3]{Goldring-Koskivirta-Strata-Hasse}}]\label{sectionszipstrata}
If $\chi\in X^*(L)$ is $\Zcal$-ample and orbitally $q$-close, then there exists $d\geq 1$ such that for all $E$-orbit $S\subset G$, the section $(h_{S,\chi})^d$ extends to $\left[E\backslash \overline{S}\right]$ with non-vanishing locus $\left[ E\backslash S\right]$.
\end{theorem}

One can state a more precise result. Let $\sigma:k\to k$ denote the inverse of the map $x\mapsto x^p$. For $w\in W$, define $w^{(0)}=e$ and by induction $w^{(r)}:={}^\sigma (w^{(r-1)}w)$ for all $r\geq 1$.

\begin{enumerate}[(i)]
\item Fix a $W$-frame $(B,T,z)$ (\Def~\ref{defWframe}). 
\item Fix $m\geq 1$ such that $T$ splits over $\FF_{q^{m}}$, where $q:=p^n$.
\item For $w\in {}^I W\cup W^J$, choose $r_w\geq 1$ such that $(w\varphi^n(z))^{(r_w)}=e$.
\item For $w\in W$, define $E_w$ as the set of roots $\alpha\in \Phi^+$ such that $ws_\alpha <w$ and $\ell(ws_\alpha) =\ell(w)-1$.
\item For $w\in {}^I W\cup W^J$, denote by $h_{w,\chi}$ the section $h_{S,\chi}$ for $S=G_w$.
\end{enumerate}

\begin{proposition}[{\cite[\Prop~3.2.1]{Goldring-Koskivirta-Strata-Hasse}}] \label{prop HI}
Let $w\in {}^I W\cup W^J$ and $\chi \in X^*(L)$. The following assertions are equivalent:
\begin{enumerate}
\item \label{item HIzip} There exists $d\geq 1$ such that $(h_{w,\chi})^d$ extends to $\left[ E\backslash \overline{G}_w\right]$ with non-vanishing locus $\left[ E\backslash G_w\right]$ (i.e to a characteristic section of $\left[ E\backslash \overline{G}_w\right]$).
\item\label{item HIformula}  For all $\alpha \in E_w$, one has:
\begin{equation}\label{eqpropHI}
n_\alpha:=\sum_{i=0}^{r_w m-1} \langle(zw^{-1})^{(i)}({}^{\sigma^i}\chi),w\alpha^\vee\rangle p^i >0.
\end{equation}
\end{enumerate}
\end{proposition}

\begin{corollary}\label{glnsections}
Let $\Zcal=(G,P,L,Q,M,\varphi^n)$ be an exponent $n$ zip datum such that $G_k\simeq GL_{n,k}$ and $P$ is a maximal parabolic. For a $\Zcal$-ample character $\lambda \in X^*(L)$, there exists $N\geq 1$ such that $\Vscr(N\lambda)$ admits characteristic sections on all zip strata.
\end{corollary}

\begin{proof}
It is clear that $\lambda$ is orbitally $q$-close, so \Th~\ref{sectionszipstrata} applies.
\end{proof}

\subsection{Sections over fine flag strata}\label{subsec sections over fine}
We fix a parabolic subgroup $B\subset P_0 \subset P$ with type $I_0\subset I$. We make the identification $X^*(\hat{E}_{P_0})=X^*(L_0)$ as above. For a character $\lambda \in X^*(L_0)$, we denote by $\Lscr_{P_0}(\lambda)$ (or $\Lscr_{\Zcal,P_0}(\lambda)$ to avoid confusion) the attached line bundle on $\GF^{(\Zcal,P_0)}\simeq \left[\hat{E}_{P_0}\backslash G\right]$.

\begin{rmk}\label{picfin}
Since $\Pic(G)$ is finite, there exists an integer $m$ such that for all line bundle $\Lscr$ on $\GF^{(\Zcal,P_0)}$, $\Lscr^{m}=\Lscr_{P_0}(\lambda)$ for some $\lambda \in X^*(L)$.
\end{rmk}

\begin{definition}
A stratification of a topological space $X$ is a decomposition $X=\bigsqcup_i X_i$ into locally closed subsets such that the closure of each $X_i$ is a union of $X_j$.
\end{definition}

\begin{definition}[Types of purity] 
A stratification of a scheme or stack  $X$ is
\begin{enumerate}
\item  \underline{principally pure} if for every stratum $S$, there exists a line bundle $\Lscr$ on $X$ and a characteristic section $f\in H^0(\overline{S},\Lscr)$ for $S$;
\item \underline{uniformly principally pure} if $\Lscr$ may be chosen independently of $S$.
\end{enumerate}
\end{definition}

This definition applies to the stack $\GF^{(G,P_0)}$ endowed with the fine stratification. By \Rmk~\ref{picfin}, we may equivalently ask for the existence of a character $\lambda \in X^*(L_0)$ such that $\Lscr_{P_0}(\lambda)$ satisfies these conditions. Note that for $P_0=P$, this definition also applies to the stack $\GZip^\Zcal$, endowed with the zip stratification.

As we mentioned in the introduction, a key problem is to understand when the zip stratification of $\GZip^\Zcal$ is principally pure. If principal purity of the zip stratification would always hold (but it is not the case), that of the fine stratification of $\GF^{(\Zcal,P_0)}$ would follow by pulling back along the map $\Psi_{P_0}:\GF^{(\Zcal,P_0)}\to \GZip^{\Zcal_{0}}$. 

We will give in \S\ref{counterex} an example where the zip stratification of $\GZip^\Zcal$ is not principally pure. We will prove below that when $\Zcal$ is maximal (in particular, if $\Zcal$ is of Hodge-type), then the zip stratification of $\GZip^\Zcal$ is uniformly principally pure. However, note that the zip datum $\Zcal_{0}$ is usually not maximal even if $\Zcal$ is.

As a consequence of \Th~\ref{sectionszipstrata}, we obtain \Th~\ref{th-intro-unif-purity} of the introduction:

\begin{theorem}\label{thmonw}
Let $(\Zcal,P_0)$ be a flagged zip datum of exponent $n$ and write $q=p^n$. If there exists a $\Zcal_0$-ample, orbitally $q$-close character of $L_0$, then the fine stratification of $\GF^{(\Zcal,P_0)}$ is uniformly principally pure.
\end{theorem}

\begin{proof}
Apply \Th~\ref{sectionszipstrata} to the stack $\GZip^{\Zcal_0}$ and pull back along $\Psi_{P_0}$.
\end{proof}

\begin{cor}
Let $G$ be a reductive group over $\FF_p$. There exists $N\geq 1$ (depending only on $G$) such that for all $n\geq N$ and all flagged zip datum $(\Zcal,P_0)$ for $G$ of exponent $n$, the fine stratfication of $\GF^{(\Zcal,P_0)}$ is uniformly principally pure.
\end{cor}

Another situation where \Th~\ref{thmonw} applies is when we start with a cocharacter datum of characteristic $0$. For large $p$, the reduction modulo $p$ will satisfy the conditions of the theorem. We will exploit this in \S\ref{sec flagsp}.

\subsection{Counter-examples}\label{counterex}

In this section, we give a counter-example to principal purity. It is rather easy to find a counter-example to uniform principal purity. For example, consider a cocharacter zip datum $(GL_{4,\FF_p},\mu)$ where
\begin{equation}
\mu:\GG_{m,k}\to GL_{4,k}, \quad z\mapsto \diag(1,z,z^2,z^3).
\end{equation}
Let $B\subset GL_{4,k}$ denote the Borel subgroup of upper-triangular matrices and $T\subset B$ the diagonal torus. The zip datum of exponent $1$ attached to $(GL_{4,\FF_p},\mu)$ is $\Zcal:=(GL_{4,\FF_p},B,T,B_-,T,\varphi)$, where $B_-$ is the opposite Borel. In this case, the zip strata of $GL_{4,k}$ coincide with the $B\times B_-$-orbits in $GL_{4,k}$. One can show that for $p=2$, there is no character $\lambda \in X^*(T)$ such that $\Vscr(\chi)$ admits characteristic sections on all zip strata, so the zip stratification is not uniformly principally pure. However, it is clearly principally pure since it coincides with the Bruhat stratification.

To find a counter-example to principal purity, consider the following case. Consider the alternating matrix

\begin{equation}\label{Jmat}
J:=\left(
\begin{matrix}
&&&&& 1\\
&&&&1&\\
&&&1&&\\
&&-1&&&\\
&-1&&&&\\
-1&&&&&
\end{matrix} \right).
\end{equation}
Denote by $G:=Sp(J)$ the symplectic group over $\FF_p$ attached to $J$ and let $\mu$ be the cocharacter of $G$ defined by
\begin{equation}\label{cocharcounter}
\mu: \GG_{m,k}\to G_k, \quad z\mapsto \diag(z,z,1,1,z^{-1}, z^{-1}).
\end{equation}

\begin{proposition}
For $p=2$, the zip stratification of $\GZip^\mu$ is not principally pure.
\end{proposition}
First we let $p\geq 2$ be an arbitrary prime number. Let $T$ be the $\FF_p$-split maximal torus given by the diagonal matrices in $G$:
\begin{equation}
T(R):=\{ \diag(x_1,x_2,x_3,x^{-1}_3,x^{-1}_2,x^{-1}_1), \ \ x_1,x_2,x_3\in R^\times \}
\end{equation}
for any $\FF_p$-algebra $R$. A Borel subgroup $B$ over $\FF_p$ containing $T$ is given by the lower-triangular matrices in $G$. For a triple $(a,b,c)\in \ZZ^3$, let $\lambda_{(a,b,c)}\in X^*(T)$ sending $\diag(x_1,x_2,x_3,x^{-1}_3,x^{-1}_2,x^{-1}_1)$ to $x^a_1 x^b_2 x^c_3$. Identify $X^*(T)$ with $\ZZ^3$ via this isomorphism. Denoting by $(e_1,e_2,e_3)$ the standard basis of $\ZZ^3$, the $T$-roots of $G$, the positive roots, the simple roots are respectively:
\begin{align*}
\Phi&:=\{e_i \pm e_j , 1\leq i \neq j \leq 3\} \cup \{ \pm 2e_i, 1\leq i \leq 3 \}\\
\Phi^+&:=\{e_i \pm e_j , 1\leq i< j \leq 3\} \cup \{ 2e_i, 1\leq i \leq 3 \}\\
\Delta&:=\{\alpha_1, \alpha_2,\alpha_3\}
\end{align*}
where $\alpha_1, \alpha_2,\alpha_3$ are defined by
\begin{equation}
\alpha_1:=e_1-e_2 \quad; \quad  \alpha_2:=e_2-e_3  \quad; \quad  \alpha_3:=2e_3.
\end{equation}
The Weyl group $W:=N_G(T)/T$ is the group of permutations $\sigma \in \Sfr_6$ satisfying $\sigma(i)+\sigma(7-i)=7$ for all $1\leq i \leq 6$. Following \cite{EkedahlGeerEO}, we denote a permutation $\sigma \in W$ by $[\sigma(1) \sigma(2) \sigma(3)]$ (these three values determine $\sigma\in W$ uniquely). The identity element $e\in W$ is $[123]$ and the simple reflections are
\begin{equation}
s_{\alpha_1}=[213] \quad; \quad s_{\alpha_2}=[132] \quad; \quad s_{\alpha_3}=[124].
\end{equation}
The longest element in $W$ is $w_0=[654]$ and has length $9$.

Define a maximal parabolic $P$ containing $B$ such that for any $\FF_p$-algebra $R$, the group $P(R)$ is the set of elements of $G(R)$ of the form:
\begin{equation}
\left(
\begin{matrix}
A & 0 & 0\\
*&B & 0\\
*&*& C
\end{matrix}
\right), \quad A,B,C \in GL_2(R).
\end{equation}
One sees immediately that such a matrix satisfies the conditions $B\in SL_2(R)$ and $C= J_0 {}^tA^{-1} J_0$ with $J_0:=\left( \begin{matrix}
0& 1\\1&0 \end{matrix}\right)$. In particular, the Levi subgroup $L$ of $P$ containing $T$ is isomorphic to $GL_{2,\FF_p}\times SL_{2,\FF_p}$ and $W_L$ identifies with the subgroup $W_L=\langle s_{\alpha_1},s_{\alpha_3} \rangle$. The longest element of $W_{L}$ is $w_{0,L}:=[214]$ (length $2$). We also define:
\begin{equation}
z:= w_0 w_{0,L}=[563].
\end{equation}
Let $Q$ be the opposite parabolic subgroup of $P$ with respect to $L$. We obtain a zip datum $\Zcal:=(G,P,L,Q,L,\varphi)$.

Here is a representation of the zip strata. For $w,w'\in W$, an arrow $w\to w'$ indicates the closure relation $G_w \subset \overline{G}_{w'}$. In this example, the order $\preceq$ coincides with the restriction of the Bruhat order $\leq$ to ${}^I W$.

\begin{center}
\begin{tikzpicture}
\node(pseudo) at (-1,0){};
\node(0) at (0,0)[]        {$[123]$};
\node(1) at (1.5,0)[]        {$[132]$};
\node(2) at (3,0.5)[]        {$[142]$};
\node(3) at (3,-0.5)[]       {$[231]$};
\node(4) at (4.5,0.5)[]        {$[153]$};
\node(5) at (4.5,-0.5)[]       {$[241]$};
\node(6) at (6,0.5)[]        {$[263]$};
\node(7) at (6,-0.5)[]       {$[351]$};
\node(8) at (7.5,0.5)[]        {$[362]$};
\node(9) at (7.5,-0.5)[]       {$[451]$};
\node(10) at (9,0)[]        {$[462]$};
\node(11) at (10.5,0)[]       {$[563]$};
 \draw[->]
 (0)      edge                 node [above]  {}     (1)
  (1)      edge                 node [above]  {}     (2)
  (1)      edge                 node [above]  {}     (3)
  (3)      edge                 node [above]  {}     (5)
  (2)      edge                 node [above]  {}     (5)
  (2)      edge                 node [above]  {}     (4)
  (4)      edge                 node [above]  {}     (7)
  (4)      edge                 node [above]  {}     (6)
  (5)      edge                 node [above]  {}     (6)
  (5)      edge                 node [above]  {}     (7)
  (6)      edge                 node [above]  {}     (8)
  (7)      edge                 node [above]  {}     (8)
  (7)      edge                 node [above]  {}     (9)
  (8)      edge                 node [above]  {}     (10)
  (9)      edge                 node [above]  {}     (10)
  (10)      edge                 node [above]  {}     (11)  ;
\end{tikzpicture}
\end{center}

In the following, we take $\chi:=\lambda_{(1,0,0)}$, which is the fundamental weight of $Q$. Consider the element
\begin{equation}
w:=[351] \in {}^I W.
\end{equation}
One has $\ell(w)=4$. The set $E_w$ of "lower neighbors" of $w$ in $W$ consists of $4$ elements:
\begin{align*}
&w_1:=[153]=ws_{\alpha_1} \quad \textrm{ with } \ \alpha_1=e_1-e_3 \\
&w_2:=[241]=ws_{\alpha_2} \quad \textrm{ with } \ \alpha_2=e_1+e_2 \\
&w_3:=[315]=ws_{\alpha_3} \quad \textrm{ with } \ \alpha_3=e_2-e_3 \\
&w_4:=[321]=ws_{\alpha_4} \quad \textrm{ with } \ \alpha_4=2e_2
\end{align*}
An easy computation shows that the multiplicities $n_\alpha$ as in \eqref{eqpropHI} are given by:
\begin{align*}
&n_{\alpha_1}=p-1\\
&n_{\alpha_2}=2p-1\\
&n_{\alpha_3}=p-2\\
&n_{\alpha_2}=p-1
\end{align*}
Hence for $p=2$, one has $n_{\alpha_3}=0$, so $g_w$ is non-vanishing along $X_{w_3}$. This shows that no power of $h_{w,\chi}$ extends to a characteristic function of $G_w$. Actually, one can show that no power of $h_{w,\chi}$ extends at all to $\overline{G}_w$.

\section{Functoriality and discrete fiber theorem}

\subsection{Discrete fiber theorem}

We consider a morphism $f:\Zcal_1 \to \Zcal_2$ of zip data of exponent $n$, with $\Zcal_i=(G_i,P_i,L_i,Q_i,M_i,\varphi^n)$, and let $I_i$ be the type of $P_i$ (for $i=1,2$). We say that $f$ has central kernel if the underlying group homomorphism $f:G_1\to G_2$ has central (scheme-theoretic) kernel.

\begin{theorem} \label{discfib}
Let $f:\Zcal_1\to \Zcal_2$ be a morphism with central kernel. Then the attached morphism of stacks
\begin{equation}
\tilde{f}:\GoneZip^{\Zcal_1}\to \GtwoZip^{\Zcal_2}
\end{equation}
has discrete fibers on the underlying topological spaces.
\end{theorem}

\begin{proof}
We first reduce the proof to the case when $f$ is a closed embedding. Set $K:=\Ker(f)$ and define $\tilde{\square}=\square/K$ for $\square=G_1,P_1,L_1,Q_1,M_1$. We get a zip datum $\tilde{\Zcal}_1:=(\tilde{G}_1,\tilde{P}_1,\tilde{L}_1,\tilde{Q}_1,\tilde{M}_1,\varphi^n)$. The natural projection $\pi: G_1 \to \tilde{G}_1$ induces a map of zip data $\Zcal_1 \to \tilde{\Zcal}_1$ and the induced map
\begin{equation}
\tilde{\pi}:\GoneZip^{\Zcal_1}\to \GtiloneZip^{\tilde{\Zcal}_1}
\end{equation}
is a bijection on the underlying topological spaces, by the proof of \cite[\Prop~3.20]{Koskivirta-Wedhorn-Hasse}. We may factor the map $f:G_1\to G_2$ as the composition $G_1\xrightarrow{\pi} \tilde{G}_1 \xrightarrow{\iota} G_2$ where $\iota$ is a closed embedding, so it suffices to prove the result for the map $\iota$. Hence we assume from now on that $f$ is a closed embedding.

\begin{lemma} \label{restrictionample}
Let $\chi_2 \in X^*(L_2)$ be a $\Zcal_2$-ample character. Then the restriction $\chi_1=\chi_2|_{L_1}\in X^*(L_1)$ is $\Zcal_1$-ample.
\end{lemma}

\begin{proof}[Proof of Lemma \ref{restrictionample}]
We first show that $f^{-1}(Q_2)=Q_1$. By definition of a map of zip data, one has $f(Q_1)\subset Q_2$, so $Q_1\subset f^{-1}(Q_2)$. In particular, $f^{-1}(Q_2)$ is a parabolic subgroup of $G_1$. Then $f^{-1}(V_2)$ is a unipotent normal subgroup of $f^{-1}(Q_2)$, hence its identity component is contained in the unipotent radical of $f^{-1}(Q_2)$. Again by definition we have $V_1\subset f^{-1}(V_2)$, so we deduce that $V_1$ is contained in the unipotent radical of $f^{-1}(Q_2)$. This implies $Q_1=f^{-1}(Q_2)$. Since $f$ is defined over $\FF_p$, we obtain an embedding $G_1/\varphi^{-n}(Q_1) \to G_2/\varphi^{-n}(Q_2)$. This proves the result.
\end{proof}

We first prove \Th~\ref{discfib} for large $n$. More precisely, one has the following lemma:

\begin{lemma}\label{lemmadisc}
Assume that there exists a $\Zcal_2$-ample character $\chi_2\in X^*(L_2)$ such that $\chi_1:=\chi_2|_{L_1}$ is $p^n$-small. Then $\tilde{f}$ has discrete fibers.
\end{lemma}

\begin{proof}[Proof of Lemma \ref{lemmadisc}]
By \Lem~\ref{restrictionample} above, $\chi_1$ is $\Zcal_1$-ample, so we may apply \Th~\ref{sectionszipstrata} to $\chi_1$. If $\tilde{f}$ doesn't have discrete fibers, then we can find two zip strata $C_1\subset G_1$ and $C'_1\subset \overline{C_1}$ mapping to the same zip stratum $C_2$ of $G_2$. By \Th~\ref{sectionszipstrata}, there exists $N\geq 1$ such that $\Vscr(N\chi_1)$ admits a section $H_1$ on $\overline{C_1}$  with non-vanishing locus $C_1$. We can also find $r\geq 1$ and a non-vanishing section $H_2$ of $\Vscr(r\chi_2)$ on $C_2$. The pull-back of $H_2^N$ to $C_1$ agrees (up to a nonzero scalar) with $H_1^r$, since they are both nonzero sections of $\Vscr(Nr\chi_1)$, and the space $H^0(C_1,\Vscr(Nr\chi_1)$ has dimension one by \cite[\Prop 1.18]{Koskivirta-Wedhorn-Hasse}. This contradicts the fact that $H_1$ vanishes on $C_1'$.
\end{proof}

We now prove the general case. Fix an integer $A$ such that all the groups $P_i,L_i,Q_i,M_i$ for $i=1,2$ are defined over $\FF_{p^A}$. For an integer $N\equiv 1 \pmod A$, we consider new zip data
\begin{align*}
\Zcal_{1,N}&=(G_1,P_1,L_1,Q_1,M_1,\varphi^{nN}) \\
\Zcal_{2,N}&=(G_2,P_2,L_2,Q_2,M_2,\varphi^{nN}).
\end{align*}
By our assumption on $N$ and $A$, the tuples $\Zcal_{1,N}$ and $\Zcal_{2,N}$ are zip data of exponent $Nn$. The map $f$ induces a morphism of zip data $f_N:\Zcal_{1,N} \to \Zcal_{2,N}$ and a morphism of stacks
\begin{equation}
\tilde{f}_N:\GoneZip^{\Zcal_{1,N}}\to \GtwoZip^{\Zcal_{2,N}}.
\end{equation}
By the parametrization \eqref{param}, there exists a finite subset $\Sigma_1 \subset G_1(k)$ (resp. $\Sigma_2\subset G_2(k)$), independent of $N$, which is a system of representatives of the $\Zcal_{1,N}$-strata in $G_1$ (resp. the $\Zcal_{2,N}$-strata in $G_2$) for all $N$. Hence $f_N$ induces a map of sets
\begin{equation}
\Psi_N:\Sigma_1 \to \Sigma_2
\end{equation}

We believe that $\Psi_N$ is independent of $N$, but we only need the following weaker result:

\begin{lemma}\label{lem1N}
Fix an integer $C\geq 1$. Assume that $x,y\in \Sigma_1$ satisfy $\Psi_1(x)=\Psi_1(y)$. Then there exists arbitrary large $N\equiv 1 \pmod{C}$ such that $\Psi_N(x)=\Psi_N(y)$.
\end{lemma}

\begin{proof}[Proof of Lemma \ref{lem1N}]
Since $\Psi_1(x)=\Psi_1(y)$, the elements $f(x)$ and $f(y)$ are in the same $\Zcal_2$-zip stratum. Hence we can find $u\in U_2$, $v\in V_2$ and $a\in L_2$ such that
\begin{equation}
f(y)=uaf(x)\varphi^n(a)^{-1}v
\end{equation}
Choose $B\geq 1$ such that $a\in G(\FF_{p^B})$. Then for $N\equiv 1 \pmod{ABC}$, we can write $f(y)=uaf(x)\varphi^{nN}(a)^{-1}v$, which shows that $\Psi_N(x)=\Psi_N(y)$.
\end{proof}

To finish the proof of \Th~\ref{discfib}, denote by $\preceq_N$ the relation (\ref{curlyorder}) for the zip datum $\Zcal_{1,N}$. Assume that $\tilde{f}=\tilde{f}_1$ doesn't have discrete fibers.  Then there exists $x,y\in \Sigma_1$ such that $x\preceq_1 y$ and $\Psi_1(x)=\Psi_1(y)$. By \Prop~\ref{relequiv}, there exists $h\in L_1$ such that $hx\varphi^{n}(x)^{-1}\leq y$. Choose $C\geq 1$ such that $h\in G(\FF_{p^C})$. Then
\begin{equation}
hx\varphi^{nN}(x)^{-1}= hx\varphi^{n}(x)^{-1} \leq y
\end{equation}
for all $N\equiv 1 \pmod{C}$, which shows that $x\preceq_N y$.
By \Lem~\ref{lem1N}, there exists arbitrary large $N\equiv 1 \pmod{C}$ such that $\Psi_N(x)=\Psi_N(y)$, hence the map $\tilde{f_N}$ does not have discrete fibers for each such $N$. But for large $N$, we can find a $\Zcal_2$-ample character of $L_2$ restricting to a $p^{Nn}$-small character of $L_1$. By \Lem~\ref{lemmadisc} the map $\tilde{f}_N$ has discrete fibers, hence a contradiction.
\end{proof}

\begin{corollary}\label{cor-f-pure}
Let $f:\Zcal_1\to \Zcal_2$ be a morphism with central kernel. If the zip stratification of $\Zcal_2$ is principally pure (resp. uniformly principally pure), then so is the zip stratification of $\Zcal_1$.
\end{corollary}

\begin{example}
We expand on the counter-example in \S\ref{counterex}. Take $G=Sp(J)$ as in (\ref{Jmat}) and $\mu:\GG_{m,k}\to G_k$ as in (\ref{cocharcounter}). There is a natural embedding $\iota:G\to GL_{6,\FF_p}$, and the cocharacter datum $(GL_{6,\FF_P},\iota\circ \mu)$ yields a zip datum $\Zcal'$ for $GL_{6,\FF_p}$. For $p=2$, \Cor~\ref{cor-f-pure} implies that the zip stratification of $\GLsixZip^{\Zcal'}$ is not principally pure.
\end{example}

\subsection{Corollaries of the discrete-fiber theorem}

\begin{cor}
Let $\Zcal=(G,P,L,Q,M,\varphi^n)$ be a zip datum and set $q:=p^n$. Assume $G$ admits an embedding $G\to GL_{N,\FF_p}$ defined over $\FF_p$. If $q\geq N$, then the zip stratification of $G$ is uniformly principally pure.
\end{cor}

\begin{proof}
By \Cor~\ref{cor-f-pure}, it suffices to consider the case $G=GL_{N,\FF_p}$. We may assume that $L$ is the Levi subgroup corresponding to a decomposition $N=N_1+...+N_r$, given by diagonal blocks of size $N_1$,...,$N_r$. We identify a character of $L$ with a tuple $(a_1,...,a_r)\in \ZZ^r$. Consider the character $\lambda=(r,r-1,...,1)$. This character is $\Zcal$-ample and orbitally $q$-close for $q\geq r$, hence the result.
\end{proof}

The shortcoming of \Th~\ref{sectionszipstrata} is that the set of $\Zcal$-ample, orbitally $q$-close characters could be empty for small values of $p$. However, we have the following result:

\begin{corollary}\label{corHT}
Let $(G,\mu)$ be a maximal cocharacter datum over $\FF_p$, and let $\chi$ be a maximal character of $L$ (\Def~\ref{charHTdef}). Then there exists $N\geq 1$ such that the line bundle $\Vscr(N\chi)$ admits characteristic sections on all zip strata of $G$. In particular, the zip stratification of $\GZip^\mu$ is uniformly principally pure.
\end{corollary}

\begin{proof}
This follows immediately from \Cor~\ref{cor-f-pure} and \Cor~\ref{glnsections}.
\end{proof}

\begin{corollary}
Let $(G,\mu)$ be a cocharacter datum over $\FF_p$ and assume $\mu$ is minuscule. For $p>2$, the zip stratification of $\GZip^\mu$ is uniformly principally pure.
\end{corollary}

\begin{proof}
Consider the adjoint representation $\Ad:G\to GL(\gfr)$ where $\gfr$ is the Lie algebra of $G$. Since $\mu$ is minuscule, the cocharacter $\Ad \circ \mu:\GG_{m,k}\to GL(\gfr)_k$ has only weights $-1,0,1$. Hence we reduce to the case of $GL(V)$ (for a finite-dimensional $\FF_p$-vector space $V$) and a zip datum whose attached parabolic $P$ is the stabilizer of a flag $0\subset V_1 \subset V_2 \subset V$. It is clear that $P$ admits an $\Zcal$-ample orbitally $3$-close character, which proves the result.
\end{proof}

\subsection{Functoriality of zip flags}
Consider two zip data $\Zcal=(G,P,L,Q,M,\varphi^n)$ and $\Zcal'=(G',P',L',Q',M',\varphi^n)$ of exponent $n$. Choose parabolic subgroups $P_0\subset P$ and $P'_0\subset P'$, together with Levi subgroups $L_0 \subset P_0$ and $L'_0 \subset P'_0$ satisfying $L_0\subset L$ and $L'_0\subset L'$. Denote by $U_0$ and $U'_0$ the unipotent radicals of $P_0$ and $P'_0$ respectively.

Let $f:\Zcal\to \Zcal'$ be a morphism of zip data satisfying the additional conditions:
\begin{enumerate}[(i)]
\item $f(L_0) \subset L'_0$
\item $f(U_0)\subset U'_0$.
\end{enumerate}
For example, these conditions are satisfied if $(P_0,L_0)$ and $(P'_0,L'_0)$ are attached to cocharacters $\mu_0:\GG_{m,k}\to G_k$ and $\mu'_0=f\circ \mu_0$ respectively (\S\ref{subsec cocharzipdata}). The map $f$ sends the subgroup $E_\Zcal\subset G_k\times G_k$ to $E_{\Zcal'}\subset G'_k\times G'_k$. Since $f(P_0)\subset P'_0$, we also have $f(\hat{E}^{\Zcal}_{P_0})\subset \hat{E}^{\Zcal'}_{P'_0}$. Furthermore, define zip data $\Zcal_{0}:=(G,P_0,L_0,Q_0,M_0,\varphi^n)$ and $\Zcal'_{0}:=(G',P'_0,L'_0,Q'_0,M'_0,\varphi^n)$ as in (\ref{zipparab}). It follows from our assumptions that $f$ induces a morphism of zip data $\Zcal_{0}\to\Zcal'_{0}$ and a commutative diagram:

$$\xymatrix@1@M=5pt{
\GZip^{\Zcal_{0}} \ar[r]^{\tilde{f}^{\Zcal_0}}  &    \GpZip^{\Zcal'_{0}} \\
\GF^{(\Zcal,P_0)} \ar[u]\ar[u]_{\Psi^{\Zcal}_{P_0}} \ar[r]^{\tilde{f}_{P_0}} \ar[d]^{\pi_{P_0}} & \GpF^{(\Zcal',P'_0)} \ar[d]^{\pi_{P'_0}} \ar[u]_{\Psi^{\Zcal}_{P_0}}  \\
\GZip^{\Zcal} \ar[r]^{\tilde{f}}  & \GpZip^{\Zcal'}
}$$

Since the diagram commutes, the map $\tilde{f}_{P_0}$ induces a map between the fibers of $\pi_{P_0}$ and $\pi_{P'_0}$. These fibers can be identified with $P/P_0$ and $P'/P'_0$ respectively, and the map $\tilde{f}_{P_0}$ then identifies with the natural map $P/P_0 \to P'/P'_0$ induced by $f$.

\section{Flag spaces and stratifications}\label{sec flagsp}

\subsection{Shimura varieties of Hodge-type}

Let $(\GG,X)$ be a Shimura datum of Hodge type with reflex field $E\subset \CC$, given by a reductive group $\GG$ over $\QQ$ and a $\GG(\RR)$-conjugacy class of morphisms $\SS\to \GG_\RR$. For $h_0\in X$, define $\mu_0:\GG_{m,\CC}\to \GG_{\CC}$ as $\mu_0(z)=h_{0,\CC}(z,1)$ using the identification $\SS_\CC\simeq \GG_{m,\CC}\times \GG_{m,\CC}$ given by $z\mapsto (z,\overline{z})$ on $\RR$-points.

We choose an integer $m\geq 1$ and a sufficiently small compact open subgroup $K\subset \GG(\AA_f)$ such that the following properties hold:

\begin{enumerate}
\item There is a reductive model $\Gcal$ of $\GG$ over $\ZZ\left[\frac{1}{m}\right]$.
\item The Shimura variety $\mathbb{S}_K :=Sh_K(\GG,X)$ admits a smooth integral model $\Scal_K$ over $R:=\Ocal_E\left[\frac{1}{m}\right]$. This model can be obtained by glueing the $\Ocal_{E,\pfr}$-models of \cite{Vasiu-Preabelian-integral-canonical-models} and \cite{Kisin-Hodge-Type-Shimura} for primes $\pfr$ of good reduction.
\item There is a cocharacter $\mu$ of $\Gcal_R$ such that $\mu_\CC$ lies in the conjugacy class of $\mu_0$.
\end{enumerate}

For each prime $\pfr$ in $R$ lying above a prime number $p$, we obtain a cocharacter datum $(G,\mu)$ where $G:=\GG \otimes \FF_p$. 

Let $f:\Ascr \to \Scal_K$ be the universal abelian scheme obtained by pull-back from the Siegel case. For a prime $\pfr$ of $R$, define $S_K:=\Scal_K \otimes R/\pfr$ the special fiber of $\Scal_K$ and $A:=\Ascr\otimes R/\pfr$. 
The de Rham cohomology $H^1_{dR}(A / S_K)$ together with its Hodge filtration and conjugate filtration naturally defines a
$G$-zip over $S_K$ (\cite[\Th~2.4.1]{ZhangEOHodge}). This induces a smooth morphism of stacks (\cite[\Th~3.1.2]{ZhangEOHodge})
\begin{equation}
\zeta:S_K \longrightarrow \GZip^\mu.
\end{equation}
The geometric fibers of $\zeta$ are the Ekedahl-Oort strata of $S_K$. See the introduction of \cite{Goldring-Koskivirta-Strata-Hasse} for further references and
remarks about $\zeta$. The Hodge line bundle of $\Scal_K$ is defined as
\begin{equation}
\omega:=f_*(\det(\Omega_{\Ascr/\Scal_K})).
\end{equation}
It is ample on $\Scal_K$. Denote by $\PP\subset \GG_{\overline{\QQ}}$ the parabolic subgroup attached to $(\GG,\mu_E)$. For a character $\lambda \in X^*(\PP)$, there is an automorphic line bundle $\Vcal_K(\lambda)$ on $\Scal_K$ attached to $\lambda$ and its special fiber coincides with the line bundle $\zeta^*(\Vcal(\lambda))$. Furthermore, there exists a character $\eta_\omega\in X^*(\PP)$ such that $\omega=\Vcal(\eta_\omega)$.

\subsection{Flag spaces}
Choose a parabolic subgroup $\PP_0\subset \PP$. We can find a finite extension $E\subset E_0 \subset \CC$ such that $\PP$ and $\PP_0$ admit models $\Pcal$ and $\Pcal_0$ over $R_0:=\Ocal_{E_0}\left[\frac{1}{m}\right]$ and such that all automorphic line bundles are defined over $R_0$.

The scheme $\Scal_K$ carries a universal $\Pcal$-torsor $\Pfr$ arising from the Hodge filtration of the de Rham cohomology of $\Ascr$, as defined in \cite[\Prop~5.3.4]{MadapusiHodgeTor} (see also \S2.1.5 in \cite{Goldring-Koskivirta-Strata-Hasse}). Then the quotient of $\Pfr$ by the group $\Pcal_0$ defines a smooth $\Pcal/\Pcal_0$-bundle
\begin{equation}
\pi : \Fcal l_K \longrightarrow \Scal_K\otimes R_0.
\end{equation}
It generalizes the flag space considered by Ekedahl and Van der Geer in \cite{EkedahlGeerEO} and by the authors in \cite{Goldring-Koskivirta-Strata-Hasse}. We call $\Fcal l_{K}$ the flag space attached to the parabolic $\PP_0$. Define $\mathbb{F}l_K:=\Fcal l_K \otimes E_0$. For a given prime $\pfr_0 \subset R_0$ lying above $\pfr\subset R$ and $p\in \ZZ$, write $Fl_K$ for the reduction of $\Fcal l_{K}$ modulo $\pfr_0$. As explained in \cite{Goldring-Koskivirta-Strata-Hasse} (10.3.3), we have a Cartesian diagram:
$$\xymatrix@1@M=5pt{
Fl_K \ar[r]^-{\zeta_{P_0}} \ar[d]_-{\pi} & \GF^{(\Zcal,P_0)} \ar[d]^-{\pi_{P_0}} \\
S_K \ar[r]_-{\zeta} & \GZip^{\Zcal}
}$$

A character $\lambda\in X^*(\PP_0)$ naturally gives rise to a line bundle $\Lcal_K(\lambda)$ on $\Fcal l_K$ (see \cite{jantzen-representations} \S5.8). On the special fiber, one has the following formula:
\begin{equation}
\Lscr_K(\lambda) = \zeta_{P_0}^*(\Lscr_{P_0}(\lambda)).
\end{equation}

\subsection{Stratifications}

We define a stratification on $F l_{K}$ by pulling back fine flag strata of $\GF^{(\Zcal,P_0)}$ along the smooth morphism $\zeta_{P_0}$. Equivalently, we define strata as the fibers of the composition
\begin{equation}
F l_{K} \xrightarrow{\zeta_{P_0}} \GF^{(\Zcal, P_0)} \xrightarrow{\Psi_{P_0}} \GZip^{\Zcal_0}.
\end{equation}

For $w\in {}^{I_0}W\cup W^{J_0}$, define the corresponding fine flag strata as:
\begin{equation}
F l_{w}:=\zeta_{P_0}^{-1}\left( Z_{P_0,w}\right)
\end{equation}
endowed with the reduced structure. Similarly, for $w\in {}^{I}W\cup  W^{J}$, define Ekedahl-Oort strata of $S_K$ by:
\begin{equation}
S_{w}:=\zeta^{-1}\left( Z_{P,w}\right).
\end{equation}
For $w\in {}^{I_0}W^{J_0}$, define the coarse flag strata by
\begin{equation}
\FF_{w}:=\zeta_{P_0}^{-1}\left( \ZZ_{P_0,w}\right).
\end{equation}
The following is a simple consequence of the results in sections \ref{sec stofzf},\ref{sec ffs}:

\begin{proposition} \
\begin{enumerate}
\item For $w\in {}^{I_0}W\cup W^{J_0}$, the fine flag stratum $F l_{w}$ is smooth of dimension $\ell(w)$.
\item For $w\in {}^{I_0} W^{J_0}$, the coarse flag stratum $\FF_{w}$ is smooth of dimension $\ell(w)+\ell(w_{0,J_0})-\ell(w_{0,I_w})-\dim(P_0)$, where $I_w:=J_0\cap w^{-1}I_0 w$.
\item The Zariski closure of a fine (resp. coarse) stratum is a union of fine (resp. coarse) strata.
\item The image by $\pi^{*}_{P_0}$ of a fine stratum is a union of Ekedahl-Oort strata.
\item For $w\in {}^{I_0} W \cup W^{J_0}$ one has $\pi^{*}_{P_0}(F l_{w})=S_{w}$ and the map $\pi:F l_{w}\to S_{w}$ is finite.
\end{enumerate}
\end{proposition}

We now prove \Th~3 of the introduction, which states that for large $p$, the fine strata $F l_{w}$ of the flag space $Fl_{K}$ are affine and define a uniformly principally pure stratification for large $p$, which generalizes \cite[\Prop~10.5(ii)]{EkedahlGeerEO} in the Siegel case.

\begin{theorem}\label{thmshim}
There exists an integer $N$ depending on $\GG,X$ and $K$ such that for all $p\geq N$, the Shimura variety $Sh_K(\GG,X)$ has good reduction at all places $\pfr|p$ and the following holds:
\begin{enumerate}
\item\label{item-thmunif} The fine stratification of $F l_{K}$ is uniformly principally pure.
\item \label{item-affine} If the closure $\overline{S}$ of a fine flag stratum $S\subset F l_{K}$  is proper, then $S$ is affine.
\end{enumerate}
\end{theorem}

\begin{proof}
Property \eqref{item-thmunif} is a simple consequence of \Th~\ref{thmonw}, so it remains to prove \eqref{item-affine}. Choose an $\Zcal_0$-ample character $\eta\in X^*(\PP_0)$. By \Rmk~\ref{rmk-ample-case}, the line bundle $\Lscr(\eta)$ is relatively ample with respect to the morphism $\pi:\Fcal l_K\to \Scal_K$. By standard arguments, there exists an integer $m\geq 1$ such that the line bundle
\begin{equation}\label{Rlb}
\Rcal:=\Lscr_K(\eta) \otimes (\pi^*\omega)^m=\Lscr_K(\eta+m\eta_\omega)
\end{equation}
is ample on $\Fcal l_K$. The character $\eta+m\eta_\omega$ is $\Zcal_0$-ample and for large prime numbers $p$, it is also $p$-small. Hence there exists an integer $N$ such that for all primes $p\geq N$, the ample line bundle $\Rcal$ admits characteristic sections on all fine flag strata, again by \Th~\ref{thmonw}. This concludes the proof.
\end{proof}

\subsection{Is there an analogue of $\GZip^{\Zcal_0}$ at the level of Shimura varieties?}
\label{sec-analogue-shimura} We have seen that the flag spaces $Fl_K$ are counterparts of the stacks $\GFZP$ in the context of Shimura varieties. Likewise, it would be interesting to construct a counterpart $S^{?}_{P_0}$ to the stack $\GZip^{\Zcal_0}$. Note that $\Zcal_0$ is usually not a zip datum arising from a cocharacter datum. The properties that should be satisfied by the scheme $S^{?}_{P_0}$ are the following: There should exist natural morphisms $\Psi^{?}_{P_0}: F l_{K} \to S^{?}_{K}$ and $\zeta^{?}:S^{?}_{P_0} \to \GZip^{\Zcal_0}$ (i.e $S^{?}_{P_0}$ should carry a universal $G$-zip of type $\Zcal_0$). These maps should give rise to a commutative diagram

$$\xymatrix@1@M=5pt{
 S^{?}_{K} \ar[r]^-{\zeta^{?}} &  \GZip^{\Zcal_0} \\
F l_{K} \ar[u]^{\Psi^{?}_{P_0}} \ar[r]^-{\zeta_{P_0}} \ar[d]_-{\pi^{*}_{P_0}} & \GF^{(\Zcal, P_0)} \ar[d]^-{\pi_{P_0}} \ar[u]_{\Psi_{P_0}}  \\
S_K \ar[r]_-{\zeta}& \GZip^{\Zcal}
}$$
such that both squares are Cartesian, that is
$$Fl_{K} \simeq S^{?}_{K} \times_{\GZip^{\Zcal_0}}\GF^{(\Zcal, P_0)}.$$
In particular, the map $\Psi^{?}_{P_0}$ should be an affine bundle with fibers isomorphic to $\AA^r$ where $r=\dim(G/P_0)$. 

As mentioned in the introduction, one of the main aspects of our program initiated in \cite{Goldring-Koskivirta-Strata-Hasse} is the analogy/connection between Griffiths-Schmid manifolds and stacks of $G$-Zips, see esp. \S1.4 of \loccit More precisely, it appears that stacks of $G$-Zips may be viewed as characteristic $p$ analogues of Mumford-Tate domains. Recall that Mumford-Tate domains are the Hodge-theoretic generalizations of period domains which admit Griffiths-Schmid manifolds as arithmetic quotients.

This analogy may shed some light on the problem of constructing a pair $(S_K^?, \Psi^?_{P_0})$ as above. 
Specifically, in \cite{Green-Griffiths-Kerr-CBMS-Texas}, Green-Griffiths-Kerr introduce and study several complex manifolds associated with a Mumford-Tate domain $D$. In the terminology and notation of \loccit (see esp. \Chap~6), these include the correspondence space $\Wcal$, the incidence variety $\Jcal$ and the cycle space $\Ucal$. For more on the cycle space, 
see \cite{Fels-Huckleberry-Wolf-book}. 

The "basic diagram" of \cite[\Chap~6]{Green-Griffiths-Kerr-CBMS-Texas} is: 
$$ \xymatrix{& \Wcal \ar[ldd] \ar[rdd] \ar[d] & \\ & \Jcal \ar[ld] \ar[rd] & \\ D& & \Ucal}.
$$
It is reminiscent of the desired diagram above for the hypothetical pair $(S_K^?, \Psi^?_{P_0})$. For example, it is suggestive that, as explained in \loccitn, the maps $\Wcal \to D$ and $\Wcal \to \Jcal$ are contractible. This should be compared with the above desideratum that $\Psi_{P_0}^?$ be an affine bundle.

\bibliographystyle{amsalpha}
\bibliography{biblio_overleaf}

\end{document}